\documentclass[a4paper,11pt]{elsarticle} 
\usepackage[left=3.5cm,right=3.5cm,top=3.5cm,bottom=3.5cm]{geometry}
\usepackage{amsmath}
\usepackage{amssymb}
\usepackage{graphicx}
\usepackage{color}
\usepackage{amssymb}
\usepackage{latexsym}
\usepackage{dsfont,pifont}
\usepackage{bbold}
\usepackage{tikz}
\usepackage{pgfplots}
\usepackage{bm,upgreek}
\usepackage{caption}
\usepackage{subcaption}
\usepackage{blkarray}
\usepackage{multirow}

\newcommand{\diag}{\mathrm{diag}}

\newcommand{\ud}{\, \mathrm{d}}  %  "d" for integrals
\newcommand{\ue}{e}

\newcommand{\vligne}[1]{\begin{bmatrix} #1 \end{bmatrix}}

\usepackage{lineno,hyperref}
\newtheorem{defn}{Definition}[section]
\newtheorem{lem}[defn]{Lemma}

\newtheorem{theo}[defn]{Theorem}

\newtheorem{rem}[defn]{Remark}

\newcommand{\bs}{\boldsymbol}
\newenvironment{proof}{
      \noindent {\bf Proof }}{\qed
      \vspace{0.25\baselineskip}
}

\newcommand{\debproof}{\begin{proof}}
\newcommand{\finproof}{\end{proof}}

\definecolor{darkmagenta}{rgb}{0.5,0,0.5}
\definecolor{hotpink}{rgb}{1,0.2,1}
\definecolor{darkgreen}{rgb}{0,0.5,0}
\definecolor{darkblue}{rgb}{0,0,0.85}
\definecolor{darkred}{rgb}{0.8,0,0}
\definecolor{mellow}{rgb}{.847, 0.72, 0.525}
\newcommand{\rouge}[1]{\textcolor{black}{#1}}

\newcommand{\bleu}[1]{\textcolor{black}{#1}}

\usepackage{pifont}

\begin{document}

\begin{frontmatter}

\title{A Discontinuous Galerkin Method for \\Approximating the Stationary Distribution of Stochastic Fluid-Fluid Processes}

%% Group authors per affiliation:
\author[uoa,acems]{Nigel Bean}
\address[uoa]{The University of Adelaide, School of Mathematical Sciences}
\address[acems]{ARC Centre of Excellence for Mathematical and Statistical Frontiers (ACEMS)}
%\fntext[myfootnote]{Since 1880.}

\author[uoa,acems]{Giang T. Nguyen\corref{mycorrespondingauthor}}
%\address[uoa]{The University of Adelaide, School of Mathematical Sciences}
%\address[acems]{ARC Centre of Excellence for Mathematical and Statistical Frontiers (ACEMS)}
%\fntext[myfootnote]{Since 1880.}

%% or include affiliations in footnotes:
\author[acems,utas]{Ma{\l}gorzata M. O'Reilly}
\address[utas]{The University of Tasmania, Faculty of Science, Engineering, and Technology}

%\ead[url]{www.elsevier.com}

\author[berlin1,berlin2]{Vikram Sunkara}
\address[berlin1]{Freie Universit\"{a}t Berlin, Department of Mathematics and Computer Science}
\address[berlin2]{Konrad-Zuse-Zentrum for Informationstechnik, Department of Numerical Mathematics}

\begin{abstract}
Introduced by Bean and O'Reilly (2014), a stochastic fluid-fluid process is a Markov processes $\{X_t, Y_t, \varphi_t\}_{t \geq 0}$, where the first fluid $X_t$ is driven by the Markov chain $\varphi_t$, and the second fluid $Y_t$ is driven by $\varphi_t$ as well as by $X_t$. That paper derived a closed-form expression for the joint stationary distribution, given in terms of operators acting on measures, which does not lend itself easily to numerical computations. 

Here, we construct a discontinuous Galerkin method for approximating this stationary distribution, and illustrate the methodology using an on-off bandwidth sharing system, which is a special case of a stochastic fluid-fluid process. 
\end{abstract}


\end{frontmatter}

%\linenumbers

%\title{A Discontinous Galerkin Method for \\Approximating the Stationary Distribution of Stochastic Fluid-Fluid Processes}
%\author{Nigel Bean\thanks{The University of Adelaide, School of Mathematical Sciences, SA 5005, Australia, ARC Centre of Excellence for Mathematical and Statistical Frontiers (ACEMS), \texttt{nigel.bean@adelaide.edu.au}} 
%\and
%Giang T. Nguyen\footnote{The University of Adelaide, School of Mathematical Sciences, SA 5005, Australia, \texttt{giang.nguyen@adelaide.edu.au}} 
%\and 
%Malgorzata M. O'Reilly\footnote{The University of Tasmania, Faculty of Science, Engineering, and Technology, TAS 7001, Australia, \texttt{Malgorzata.OReilly@utas.edu.au}}
%\and
%Vikram Sunkara\footnote{Freie Universit\"{a}t Berlin, Department of Mathematics and Computer Science, D-14195  Berlin, Germany. \texttt{sunkara@mi.fu-berlin.de}} \footnote{Department of Numerical Mathematics, Konrad-Zuse-Zentrum for Informationstechnik, D-10557 Berlin, Germany.}}
%\date{{\today}}
%\maketitle

\section{Introduction}

A stochastic fluid process $\{X_t, \varphi_t\}_{t \geq 0}$ is a two-dimensional Markov process, where the phase $\varphi_t$ is a continuous-time Markov chain on a finite state space~$\mathcal{S}$, and the fluid $X_t$ varies linearly at rate $c_{\varphi_t}$. A subset of Markov additive processes, stochastic fluids have been well-analysed in the past two decades. There have been two recent generalisations of stochastic fluid processes to a higher dimension: Miyazawa and Zwart~\cite{mz2012} analysed discrete-time multidimensional Markov additive processes, and Bean and O'Reilly~\cite{bo2014} studied the so-called \emph{stochastic fluid-fluid process}, the latter is our focus in this paper. 

A stochastic fluid-fluid is a Markov process $\{X_t, Y_t, \varphi_t\}_{t \geq 0}$, where the phase $\varphi_t$ is still a Markov chain on a finite state space $\mathcal{S}$;  $X_t \in (-\infty,\infty)$ is the first fluid, which varies linearly at rate~$c_{\varphi_t}$
	\begin{align*} 
	 	X_t := X_0 + \int_0^t c_{\varphi_s} \ud s; 
	\end{align*} 
and $Y_t$ is the second fluid, which varies linearly at rate $r_{\varphi_t}(X_t)$: 
	\begin{align*} 
		Y_t := Y_0 + \int_0^t r_{\varphi_s}(X_s) \ud s.
	\end{align*} 
	% 
	%Thus, $Y_t$ is a Markov process on an uncountable state space $\mathcal{S} \times \mathds{R}$. 
	As the classic fluid process $\{X_t, \varphi_t\}_{t \geq 0}$ is used extensively in many areas, such as insurance and environmental modelling, it is clear that stochastic fluid-fluid models have even a wider range of applicability. 

An example of application for a stochastic fluid fluid is the modelling of growth and bleaching of coral reefs, as described in~\cite{bo2014}. In this process, we can model the density of symbiotic zooxanthellae at time $t$ by $X_t$, with the positive rates $c_i$ corresponding to the growth of the zooxanthellae, the negative rates to the bleaching. If the density $X_t$ is below a certain threshold~$x$, the coral cannot store lipids until the density increases again past this level. During the time $X_t \in (0,x)$, the coral relies on stored lipids, modelled as $Y_t$, and dies when the latter runs out, that is, $Y_t = 0$. 

While the analyses in~\cite{mz2012, bo2014} are markedly different, both papers drew inspiration from Neuts' matrix-analytic approach \cite{neuts81, lr99} to obtain the limiting behaviour of these processes, working with operators on function spaces instead of matrices. Thus, their closed-form expressions for the limiting distributions (\cite[Theorem~4.1]{mz2012}, \cite[Theorem~2]{bo2014}) are given in terms of operators acting on measures, which are not immediately amenable to numerical computations for real-life applications. One way to numerically handle operators on function spaces is to construct approximations of the operators. To this end, there exist numerical procedures such as finite difference, finite volume, finite element, and discontinuous Galerkin (DG) methods~\cite{c99}. The operators arising from fluid-fluid processes are assumed to be acting on a function space of smooth probability densities. The choice of an approximation method should reflect these properties in its solutions. In the DG method, the conservation of probability mass and  local smoothness can be captured in the approximations~\cite{c99}. 

In this paper, we construct a discontinuous Galerkin method to approximate the joint stationary distribution of a stochastic fluid-fluid process. 
We numerically illustrate the effectiveness of the methodology using an on-off bandwidth-sharing system of two processors~\cite{lnp13}. In this example, inputs into the processors, $X_t$ and $Y_t$, are turned on and off by a Markov chain, $\varphi_t$; the combined output capacity is fixed and allocated according to the workload of the first, high-priority, processor~$X_t$. Latouche \emph{et al.}~\cite{lnp13} evaluated the marginal limiting distribution of the first processor $X_t$, and provided bounds for the marginal limiting distribution of the workload of the second processor $Y_t$. We verify our DG approximations by comparing them against Monte Carlo simulations of the system, against analytical results obtained in~\cite{lnp13}, and against our intuitive understanding of the system dynamics. In all considered cases, we find the approximations to be accurate.

%\intrusion{Say something about how great these approximations are.}

The paper is organised as follows. In Section~\ref{sec:prelim}, we give relevant background to present the joint stationary distribution of a stochastic fluid-fluid process. We construct in Section~\ref{sec:numframe} a discontinuous Galerkin method to approximate the stationary distribution, and include numerical experiments in Section~\ref{sec:numexp}. 

\section{Preliminaries} 
	\label{sec:prelim}
	
Consider a stochastic fluid-fluid process $\{X_t, Y_t, \varphi_t\}_{t \geq 0}$. We assume that $X_t, Y_t \in [0,\infty)$ and that there is a regulated boundary at level $0$ for both buffers: 
	\begin{align*} 
		\frac{\ud}{\ud t} X_t & := \max\{0, c_i\} \quad \mbox{if } X_t = 0 \mbox{ and } \varphi_t = i, \\
          \frac{\ud}{\ud t} Y_t & := \max\{0, r_i(x)\} \quad \mbox{if } Y_t = 0, X_t = x \mbox{ and } \varphi_t = i, 	
	\end{align*} 
for $i \in \mathcal{S}$. Let $T$ be the irreducible generator for the finite Markov chain $\varphi_t$. 
%  on $\mathcal{S} = \{1, \ldots, L\}$. 
We denote by $C := \diag(c_i)_{i \in \mathcal{S}}$ the diagonal fluid-rate matrix for $X_t$, and $R(x) := \diag(r_i(x))_{i \in \mathcal{S}}$ the diagonal fluid-rate matrix for $Y_t$. For the remainder of this section, we summarize the findings of~\cite{bo2014} on the joint stationary distribution of $\{X_t, Y_t, \varphi_t\}_{t \geq 0}$.

Let $\mathcal{F}$ be the state space of $X_t$, so $\mathcal{F} = [0,\infty)$. For each Markovian state $i \in \mathcal{S}$, we partition $\mathcal{F}$ according to the rates of change $r_i(\cdot)$ for the second fluid $Y_t$: $\mathcal{F} := \mathcal{F}^{+}_i \cup \mathcal{F}^{-}_i \cup \mathcal{F}^{0}_i,$  
where 
		\begin{align} 
				\label{eqn:Fi1}
			\mathcal{F}^{+}_i & := \{u \in \mathcal{F} : r_i(u) > 0\},  \\
				\label{eqn:Fi2}
			\mathcal{F}^{-}_i & :=  \{u \in \mathcal{F}:  r_i(u) < 0\}, \\
				\label{eqn:Fi3}
			\mathcal{F}^{0}_i & := \{u \in \mathcal{F}: r_i(u) = 0\}.
		\end{align} 
For all $i \in \mathcal{S}$, the functions $r_i(\cdot)$ are assumed to be sufficiently well-behaved that $\mathcal{F}^{m}_i$, $m \in \{+, -, 0\}$, is a finite union of intervals and isolated points. Moreover, define $\mathcal{S}_{+} := \{ i \in \mathcal{S}: \mathcal{F}^+_i \neq \varnothing\}$, $\mathcal{S}_{-} := \{ i \in \mathcal{S}: \mathcal{F}^{-}_i \neq \varnothing\}$, and $\mathcal{S}_{0} := \{i \in \mathcal{S}: \mathcal{F}^{0}_i \neq \varnothing\}$.  

We assume that the process $\{X_t, Y_t, \varphi_t\}$ is positive recurrent, in order to guarantee the existence of the joint stationary density operator $\boldsymbol{\pi}(y) = (\pi_i(y))_{i \in \mathcal{S}}$ and the joint stationary mass operator $\boldsymbol{p} = (p_i)_{i \in \mathcal{S}}$, where for $\mathcal{A} \subset \mathcal{F}$
	\begin{align} 
		\label{eqn:jointpi} 
	\pi_i(y)(\mathcal{A}) & := \lim_{t \rightarrow \infty} \frac{\partial}{\partial y} \mathds{P}\left[X_t \in \mathcal{A}, Y_t \leq y, \varphi_t = i\right], \\
		\label{eqn:jointmass}
		p_i(\mathcal{A}) & := \lim_{t \rightarrow \infty}  \mathds{P}[X_t \in \mathcal{A}, Y_t = 0, \varphi_t = i].
	\end{align} 

The determination of $\bs{\pi}(y)$ involves two important matrices of operators, $\mathbb{B}$ and $\Psi$. Intuitively, for a set $\mathcal{A} \in \mathcal{F}$ and a measure vector $\boldsymbol{\mu} = (\mu_i)_{i \in \mathcal{S}}$, $\bs{\mu}\mbox{e}^{\mathbb{B}t}(\mathcal{A})$ gives the conditional probability of $X_t \in \mathcal{A}$, and $\boldsymbol{\mu} \Psi (\mathcal{A})$ the conditional probability of $Y_t$ returning to level zero and doing so when $X_t \in \mathcal{A}$, given that the initial distribution is $\boldsymbol{\mu}$.

\subsection{Matrix $\mathbb{B}$ of Operators}
\label{subsec:B_operators}

Let $\mathcal{M}(\mathcal{S} \times \mathds{R}_{+})$ be the set of integrable complex-valued Borel measures on the Borel $\sigma$-algebra $\mathcal{B}_{\mathcal{S} \times \mathds{R}_{+}}$. For $\bs{\mu} \in \mathcal{M}(\mathcal{S} \times \mathds{R}_{+})$, we can write 
	\begin{align*}
		\bs{\mu} = \vligne{(\mu^{+}_i)_{i \in \mathcal{S}_{+}} & (\mu^{-}_i)_{i \in \mathcal{S}_{-}} & (\mu^{0}_i)_{i \in \mathcal{S}_{0}}},
	\end{align*}
	where $\mu^{\ell}_i \in \mathcal{M}(\mathcal{F}^{\ell}_{i})$, the set of integrable complex-valued Borel measures on $\mathcal{B}_{\mathcal{F}_{i}^{\ell}}$, and  
	\begin{align*}
		 \mu_i^{\ell}(\mathcal{A}) := \mu_i^{\ell}(\mathcal{A} \cap \mathcal{F}_i^{\ell}) \quad \mbox{for } \mathcal{A} \subset \mathcal{F}. 
	\end{align*} 
%	\begin{align}
%	\boldsymbol{\mu}(\mathcal{A}) := \sum_{\ell \in \{+, -, 0\}}\sum_{i \in \mathcal{S}} \mu_i^{\ell}(\mathcal{A}\cap \mathcal{F}_i^{\ell})
%	\end{align} 
% 
We denote by $\mathbb{V}(t)$ the matrix of operators
%, $\mathbb{V}(t): \mathcal{M}(\mathcal{S} \times \mathds{R}_{+}) \mapsto \mathcal{M}(\mathcal{S} \times \mathds{R}_{+})$, consisting of operators 
$\mathbb{V}_{ij}^{\ell m}(t): \mathcal{M}(\mathcal{F}_i^{\ell}) \mapsto \mathcal{M}(\mathcal{F}^{m}_j)$, $i \in \mathcal{S}_{\ell}, j \in \mathcal{S}_m$ and $\ell, m \in \{+, - ,0\}$, which are defined for $\mathcal{A} \subset \mathcal{F}_j^m$ as follows:
\begin{align} 
	\mu_i^{\ell}\mathbb{V}_{ij}^{\ell m}(t)(\mathcal{A}) := \int_{x \in \mathcal{F}^{\ell}_i} \ud \mu_i^{\ell}(x) \mathds{P}[\varphi_t = j, X_t \in \mathcal{A} | \varphi_0 = i, X_0 = x],
\end{align} 
the probability of the process $\{X_t, \varphi_t\}$ being in the destination set $(\mathcal{A}, j)$ at time $t$, given that it starts in $(\mathcal{F}^{\ell}_i, i)$ according to the measure $\mu_i^{\ell}$. We can write $\mathbb{V}(t)$ in terms of its infinitesimal generator $\mathbb{B}$ as 
	\begin{align}
		\label{eqn:VB} 
		\mathbb{V}(t) = \ue^{\mathbb{B}t} \quad \mbox{for } t \geq 0.  
	\end{align} 
	
	For $\mathbb{V}(t)$ and $\mathbb{B}$, as well as other operators of the same dimensions to be introduced later in the paper, the operators are partitioned according to $\{+, -, 0\}$; for example,
	\begin{align*}
		\mathbb{B} = \left[\begin{array}{lll}
			\mathbb{B}^{++} & \mathbb{B}^{+-} & \mathbb{B}^{+0} \\ 
		    \mathbb{B}^{-+} & \mathbb{B}^{--} & \mathbb{B}^{-0} \\ 
			\mathbb{B}^{0+} & \mathbb{B}^{0-} & \mathbb{B}^{00} \\ 		
			\end{array}\right],  
	\end{align*} 
	where each block $\mathbb{B}^{\ell m}$ is the $|\mathcal{S}_{\ell}| \times |\mathcal{S}_m|$ matrix of operators $\mathbb{B}^{\ell m}_{ij}$. We note that $\mathbb{V}(t)$ forms a strongly continuous semigroup on the set of measures that are absolutely continuous with respect to Lebesgue measure and have analytic densities as well as possibly a point mass at zero in certain phases~\cite{bo2014}. 	
%	\intrusion{$\# 2$ Nigel's question: ``Is ``only on'' true? Or is the domain of the generator?''}
%	
	Thus, we restrict the domain of $\mathbb{V}(t)$ to 
%$\mathcal{M}^*_0(\mathcal{S} \times \mathds{R}_{+})$, 
the set of such measures, %The subscript $0$ is a reminder that these measures might have point masses at $x = 0$, corresponding to states $i \in \mathcal{S}$ with fluid rates $c_i \leq 0$ for $X_t$.
%
%
%Consider the partition $\mathcal{S} := \mathcal{S}_{u} \cup \mathcal{S}_{d} \cup \mathcal{S}_{z}$, where $\mathcal{S}_{u} = \{i \in \mathcal{S}: c_i > 0\}$, $\mathcal{S}_{d} = \{i \in \mathcal{S}: c_i < 0\}$, and $\mathcal{S}_{z} = \{i \in \mathcal{S}: c_i = 0\}$. 
%Let $\mathcal{P} \subset \mathcal{S}$ denote the set of states for which a point mass at $x = 0$ can exist, then 
%	% 
%\begin{align*} 
%	\mathcal{P} := \{ i \in \mathcal{S}: c_i < 0\} \cup \{i \in \mathcal{S}: c_i = 0, \; -\vligne{\bs{0} & \bs{1}} \left[\begin{array}{cc} 
%	T_{zz} & T_{zd} \\ 
%	T_{dz} & T_{dd}
%	\end{array}\right]^{-1} \bs{e}_{i}^{\top} > 0\},
%\end{align*} 
%	% 
%	where $\bs{0}$ is $1 \times |\mathcal{S}_{z}|$, $\bs{1}$ is $1 \times |\mathcal{S}_{d}|$, and $\bs{e}_i$ is a row vector of zeros with a unit in the $i$th position. 
%\intrusion{A reference for the domain on which $\mathbb{V}(t)$ forms a strongly continuous semigroup.}
and write  
	\begin{align*}
%		\mu_i^{\ell}(\mathcal{A}) = \int_{x \in \mathcal{A}} \ud \mu_i^{\ell}(x) = \int_{x \in \mathcal{A}} \nu_i^{\ell}(x) \ud x + \mathds{1}_{\{i \in \mathcal{P}\}}\mathds{1}_{\{0 \in \mathcal{A}\}}p_i^{\ell},  
		\mu_i^{\ell}(\mathcal{A}) = \int_{x \in \mathcal{A}} \ud \mu_i^{\ell}(x) = \int_{x \in \mathcal{A}} \nu_i^{\ell}(x) \ud x + \mathds{1}_{\left\{0 \in \mathcal{A}\right\}}p_i^{\ell},
	\end{align*} 
	where $\nu_i^{\ell}$ is the associated density and $p^{\ell}_i$ is the probability mass at the boundary $0$ when $\varphi_t = i$, for $\ell \in \{+, -, 0\}$ and $i \in \mathcal{S}$. 
	%Clearly, $p^{\ell}_i = 0$ when such a point mass does not exist for some $i$ and $\ell$. 

For simplicity, from here on we assume a set $\mathcal{A}$ is an interval, which might or might not include its end points, that is, $\mathcal{A} \in \{(u,v), (u,v], [u,v), [u,v]\}$. By \cite[Lemma~3]{bo2014}, the operators $\mathbb{B}_{ij}^{\ell m}: \mathcal{M}(\mathcal{F}^{\ell}_i) \mapsto \mathcal{M}(\mathcal{F}^{m}_j)$, $\ell, m \in \{+, -, 0\}$ and $i, j \in \mathcal{S}$, are given as follows. We include also brief probabilistic interpretations of the terms (see remarks in \cite{bo2014} for more details). Note that in all of the following cases, we assume $\mathcal{A} \subset \mathcal{F}_j^m$; this is without loss of generality, as for any set $\mathcal{G}$ and a measure $\boldsymbol{\mu}$ in our chosen domain,
	\begin{align*}
		[\boldsymbol{\mu}\mathbb{B}(\mathcal{G})]_{\ell, i} := \sum_{m \in \{+, -, 0\}}\sum_{j \in \mathcal{S}} \mu_{i}^{\ell}\mathbb{B}^{\ell m}_{ij}(\mathcal{G} \cap \mathcal{F}_j^m).
	\end{align*} 
	
	\begin{enumerate}
		\item[\textbf{Case 1.}] When $i \neq j$, 
	\end{enumerate} 
		\begin{align*} 
		\mu_i^{\ell} \mathbb{B}_{ij}^{\ell m}(\mathcal{A})  = \left\{ \begin{array}{l} 
		T_{ij} \left[\displaystyle\int_{x \in \mathcal{A} \cap \mathcal{F}^{\ell}_{i}} \nu_{i}^{\ell}(x) \ud x + p_i^{\ell}\mathds{1}_{\left\{0 \in \mathcal{A}\cap \mathcal{F}^{\ell}_i\right\}}\right] \quad \mbox{for } c_j \leq 0, \\
		\vspace*{-0.2cm} \\
		T_{ij} \left[\displaystyle\int_{x \in \mathcal{A} \cap \mathcal{F}^{\ell}_{i}} \nu_{i}^{\ell}(x) \ud x + p_i^{\ell}\mathds{1}_{\left\{0 \in \mathcal{F}^{\ell}_i\right\}}\mathds{1}_{\left\{u = 0, \; v > 0\right\}}\right] \quad \mbox{for } c_j > 0. 
		\end{array} \right. 
		\end{align*}

	Case $1$ represents when there is a stochastic jump from state~$i$ to state~$j \neq i$, which happens with rate $T_{ij}$. The integral represents the probability mass of the intersection of the initiating domain $\mathcal{F}^{\ell}_i$ and the destination set $\mathcal{A}$. If  $c_j \leq 0$, then the point mass $p_i^{\ell}$ is preserved after the change of phase. If $c_j > 0$, then the point mass $p_i^{\ell}$ disperses into the density in state~$j$, and is captured only if $\mathcal{A}$ has an upper bound strictly greater than $0$. Note that in this case $0$ does not have to be in $\mathcal{A}$, it only has to be in the closure of $\mathcal{A}$. 
	
	\begin{enumerate}
		\item[\textbf{Case 2.}] When $i = j$ and $\ell \neq m$, 
	\end{enumerate} 
	\begin{align*} 
		& \mu_i^{\ell} \mathbb{B}_{ii}^{\ell m} (\mathcal{A}) \\
		& =  \left\{ \begin{array}{l}
		-c_i \nu_i^{\ell}(v)\mathds{1}_{\{u \neq v\}}\mathds{1}_{\left\{v \in \partial_{L \backslash R}[\mathcal{F}^{\ell}_i]\right\}} - c_i\nu_i^{\ell}(0)\mathds{1}_{\{v = 0\}}\mathds{1}_{\left\{0 \in \partial_{L\backslash R}[\mathcal{F}^{\ell}_i]\right\}} \quad \mbox{for } c_i < 0, \\
		\vspace*{-0.2cm} \\
		c_i \nu_i^{\ell}(u)\mathds{1}_{\{u \neq v\}}\mathds{1}_{\left\{u \in \partial_{R\backslash L}[\mathcal{F}^{\ell}_i]\right\}} \quad \mbox{for } c_i > 0, 
		\end{array} \right. 
	\end{align*} 
	where $\partial_{L\backslash R}[\mathcal{G}]$ denotes the left boundary point that mustn't also be the right boundary point of the closure of the set $\mathcal{G}$, and similarly for $\partial_{R\backslash L}[\mathcal{G}]$. 
	
	Case $2$ represents a drift from $\mathcal{F}^{\ell}_i$ to $\mathcal{F}^{m}_i$. When neither $\mathcal{F}^{\ell}_i$ nor $\mathcal{A}$ is an isolated point, there is a transfer of density from one set to another through the relevant endpoints. When $\mathcal{A} = \{0\}$, $c_i < 0$, and $0$ is the left endpoint of the closure of $\mathcal{F}_i^{\ell}$ (which can't be an isolated point itself), then there is also an accumulation of point mass.
	
	\begin{enumerate}
		\item[\textbf{Case 3.}] When $i = j$ and $\ell = m$, 
	\end{enumerate} 	 
	\begin{align*}
		& \mu_i^{\ell} \mathbb{B}_{ii}^{\ell \ell}(\mathcal{A}) \\
		& = 
		\left\{ \begin{array}{l} 
			T_{ii}\left[\displaystyle \int_{x \in \mathcal{A}} \nu_i^{\ell} (x)\ud x + p_i^{\ell}\mathds{1}_{\left\{0 \in \mathcal{A}\right\}}\right] 
			+ \mathds{1}_{\{u \neq v\}} \left[c_i \nu_i^{\ell}(u) - c_i\nu_i^{\ell}(v) \mathds{1}_{\left\{v \not\in \partial_R[\mathcal{F}_i^{({\ell})}]\right\}}\right] \\
		\vspace*{-0.2cm} \\			
			 \hspace*{1cm} - c_i \nu_i^{\ell}(0)\mathds{1}_{\{0 \in \mathcal{A}\}}\mathds{1}_{\left\{0 \not\in \partial_R[\mathcal{F}_i^{({\ell})}]\right\}} \quad \mbox{for } c_i < 0, \\
		\vspace*{-0.2cm} \\
			T_{ii}\left[\displaystyle\int_{x \in \mathcal{A}} \nu_i^{\ell} (x)\ud x 
			% + p_i^{\ell}\mathds{1}_{\{0 \in \mathcal{A}\}}
			\right] + \mathds{1}_{\{u \neq v\}}\left[c_i\nu_i^{\ell}(u) \mathds{1}_{\left\{u \not\in \partial_L[\mathcal{F}^{\ell}_i]\right\}} - c_i \nu_i^{\ell}(v)\right]  \\
			\vspace*{-0.2cm} \\
			\hspace*{1cm}
			 \quad \mbox{for } c_i > 0, 
			\end{array}\right.
	\end{align*} 
where $\partial_R[\mathcal{G}]$ denotes the right boundary point of the closure of $\mathcal{G}$, and similarly for $\partial_L[\mathcal{G}]$. Case $3$ represents stochastic jumps out of state $i$ and drift across $\mathcal{F}^{\ell}_i$.

\subsection{Matrix $\Psi$ of Operators}

We denote by $\Psi(s)$ the $|\mathcal{S}_{+}| \times |\mathcal{S}_-|$ matrix of operators recording the Laplace-Stieltjes transforms of the time for $Y_t$ to return, for the first time, to the initial level of zero. Define the stopping time $\theta(y):= \inf \{t > 0: Y_t = y\}$ to be the first time $Y_t$ hits level $y$, then each component $\Psi_{ij}(s): \mathcal{M}(\mathcal{F}^{+}_i) \mapsto \mathcal{M}(\mathcal{F}^{-}_j), i \in \mathcal{S}_{+}$ and $j \in \mathcal{S}_{-}$, is given by  
\begin{align*} 
	& \mu_i^{+}\Psi_{ij}(s) (\mathcal{A}) \\
	& := \int_{x \in \mathcal{F}^+_i} \ud \mu_i^+(x)
	 \mathds{E}\left[\ue^{-s\theta(0)}\mathds{1}_{\left\{\varphi_{\theta(0)} = j, \; X_{\theta(0)} \in \mathcal{A}\right\}} | X_0 = x, Y_0 = 0, \varphi_0 = i\right].
\end{align*} 

Let $b(t) := \int_0^t \left| r_{\varphi_z}(X_z) \right|  \ud z$ be the total unregulated amount of fluid that has flowed into or out of the second buffer $Y_t$ during $[0,t]$, and let $\omega(y) := \inf \{t > 0: b(t) = y\}$ be the first time this accumulated in-out amount hits level $y$. We denote by $\mathbb{U}(y,s)$ the matrix of operators recording the Laplace-Stieltjes transforms of $\omega(y)$: 
	\begin{align*} 
		\mathbb{U}(y,s) = \left[\begin{array}{cc} 
			\mathbb{U}^{++}(y,s) & \mathbb{U}^{+-}(y,s) \\
			\mathbb{U}^{-+}(y,s) & \mathbb{U}^{--}(y,s) \\			
		\end{array}\right], 
	\end{align*} 
	where $\mathbb{U}^{\ell m}$ is the $|\mathcal{S}_{\ell}| \times |\mathcal{S}_{m}|$ matrix of operators $\mathbb{U}^{\ell m}_{ij}$, for $y > 0$, $s \in \mathds{C}$, and Re$(s) \geq 0$. Each operator $\mathbb{U}_{ij}^{\ell m}(y,s): \mathcal{M}(\mathcal{F}_i^{\ell}) \mapsto \mathcal{M} (\mathcal{F}^m_j)$, for $\ell, m \in \{+, -\}$ and $i, j \in \mathcal{S}$, is given by 
	\begin{align*} 
		\mu_i^{\ell}\mathbb{U}_{ij}^{\ell m} (y,s) (\mathcal{A}) := \int_{x \in \mathcal{F}^{\ell}_i} \ud \mu_i^{\ell}(x) \mathds{E}\left[\ue^{-s\omega(y)}\mathds{1}_{\left\{\varphi_{\omega(y)} = j, \; X_{\omega(y)} \in \mathcal{A}\right\}} | \varphi_0 = i, X_0 = x\right].
	\end{align*} 
	
We can write 
	\begin{align*}
		\mathbb{U}(y,s) = \ue^{\mathbb{D}(s)y},
	\end{align*} 
where $\mathbb{D}(s)$ is the infinitesimal generator of the strongly continuous semigroup $\mathbb{U}(\cdot, y)$. Lemma~$4$ of \cite{bo2014} gives the following expression for $\mathbb{D}(s)$. 

\begin{lem} 
	For $y \geq 0$, $s \in \mathds{C}$ with \emph{Re}$(s) \geq 0$, $\ell, m \in \{+, -\}$, and $i \in \mathcal{S}_{\ell}, j \in \mathcal{S}_m$, 
	\begin{align*}
		\mathbb{D}_{ij}^{\ell m}(s) = [\mathbb{R}^{\ell}(
		\mathbb{B}^{\ell m} - s\mathbb{I} + \mathbb{B}^{\ell 0}(s \mathbb{I} - \mathbb{B}^{00})^{-1}\mathbb{B}^{0m})]_{ij}, 
	\end{align*} 
	where $\mathbb{R}^{\ell} := \diag(\mathbb{R}_i^{\ell})_{i \in \mathcal{S}_{\ell}}$ is a diagonal matrix of operators $\mathbb{R}_i^{\ell}$ given by 
	\begin{align*} 
		{\mu}_i^{\ell}\mathbb{R}_i^{\ell}(\mathcal{A}) := \int_{x \in \mathcal{A} \cap \mathcal{F}_i^{\ell}} \frac{1}{r_i(x)}\ud \mu_i^{\ell}(x).
	\end{align*} 
\end{lem}

By \cite[Theorem~1]{bo2014}, $\Psi(s)$ has the following characterisation. 

\begin{theo} 
	\label{theo:Psi} 
	For \emph{Re}$(s) \geq 0$, $\Psi(s)$ satisfies the  equation: 
	\begin{align*} 
		\mathbb{D}^{+-}(s) + \Psi(s)\mathbb{D}^{-+}(s)\Psi(s) + \mathbb{D}^{++}(s)\Psi(s) + \Psi(s)\mathbb{D}^{--}(s) = 0. 
	\end{align*} 
	Furthermore, if $s$ is real then $\Psi(s)$ is the minimal nonnegative solution. 
\end{theo} 

\subsection{Stationary Distribution} 

Let $\Psi := \Psi(0)$. We define $\theta_n := \inf\{t \geq \theta_{n - 1}: Y_t = 0\}$, for $n \geq  2$, to be the sequence of hitting times to level $0$ of $Y_t$, with $\theta_1: = \theta(0)$. Consider a discrete-time Markov process $\{X_{\theta_n}, \varphi_{\theta_n}\}_{n \geq 1}$, and for $i \in \mathcal{S}_{-}$ define the measure $\xi_i$ as follows 
	\begin{align*}
		\xi_i(\mathcal{A}) := \lim_{n \rightarrow \infty} \mathds{P}\left[X_{\theta_n} \in \mathcal{A}, \varphi_{\theta_n} = i\right].
	\end{align*} 
	
By \cite{bo2014}, the vector of measures $\boldsymbol{\xi} := (\xi_i)_{i \in \mathcal{S}}$ satisfies the following set of equations 
 	\begin{align}
		\vligne{\boldsymbol{\xi}  & \boldsymbol{0}}\left(-\left[\begin{array}{ll} 
			\mathbb{B}^{--} & \mathbb{B}^{-0} \\
			\mathbb{B}^{0-} & \mathbb{B}^{00} 
		\end{array} \right]\right)^{-1}\left[\begin{array}{l} 
			\mathbb{B}^{-+} \\ 
			\mathbb{B}^{0+}
		\end{array} \right]\Psi & = \boldsymbol{\xi}, \\ 
		\sum_{i \in \mathcal{S}_{-}}\xi_i(\mathcal{F}^-_i) & = 1. 
	\end{align} 
	
We reproduce Theorem 2 of \cite{bo2014} below, which gives the joint stationary distribution of $\{X_t, Y_t, \varphi_t\}$. Recall that the joint stationary density operator $\boldsymbol{\pi}(y) = (\pi_i(y))_{i \in \mathcal{S}}$ for $\{X_t, Y_t, \varphi_t\}$ and the joint stationary mass operator $\boldsymbol{p} = (p_i)_{i \in \mathcal{S}}$ are defined by~\eqref{eqn:jointpi} and \eqref{eqn:jointmass}, respectively. We can write 
	\begin{align*} 
		\bs{\pi}(y) = \vligne{\bs{\pi}^{+}(y) & \bs{\pi}^{-}(y) & \bs{\pi}^{0}(y)} = \vligne{(\pi^{+}_i(y))_{i \in \mathcal{S}_+} & (\pi^{-}(y)_{i \in \mathcal{S}_{-}} & (\pi^{0}(y))_{i \in \mathcal{S}_{0}}},
	\end{align*} 
	where  
	\begin{align*} 
		\pi_i^{\ell}(y)(\mathcal{A}) = \pi_i^{\ell}(y)(\mathcal{A}\cap \mathcal{F}_i^{\ell}) \quad \mbox{for } \mathcal{A} \subset \mathcal{F}.
	\end{align*} 
\begin{theo} 
	\label{theo:density} 
The density $\boldsymbol{\pi}^{\ell}(y)$, for $\ell \in \{+,-,0\}$ and $y > 0$, and the probability mass $\bs{p}^{m}$, for $m \in \{-,0\}$, satisfy the following set of equations:
	\begin{align} 
	& \; \bs{\pi}^{0}(y) = \vligne{\bs{\pi}^{+}(y) & \bs{\pi}^{-}(y)}\left[\begin{array}{l} \mathbb{B}^{+0} \\ \mathbb{B}^{-0} \end{array} \right]\left(-\mathbb{B}^{00}\right)^{-1}, \\
	&  \vligne{\bs{\pi}^{+}(y) & \bs{\pi}^{-}(y)} = \vligne{\bs{p}^{-} & \bs{p}^{0}}\left[\begin{array}{l} \mathbb{B}^{-+} \\ \mathbb{B}^{0+} \end{array} \right]\vligne{\ue^{\mathbb{K}y} & \ue^{\mathbb{K}y}\Psi}\left[\begin{array}{cc} \mathbb{R}^{+} & 0 \\ 0 & \mathbb{R}^{-}\end{array}\right], \\
	&  \vligne{\bs{p}^{-}  & \bs{p}^{0}} = \alpha \vligne{\bs{\xi} & \bs{0}} 
	\left(-\left[\begin{array}{ll} 
		\mathbb{B}^{--} & \mathbb{B}^{-0} \\
		\mathbb{B}^{0-} & \mathbb{B}^{00} 
		\end{array} \right] \right)^{-1},  \label{eqn:mass}\\
	& \sum_{\ell \in \{+,-,0\}}\sum_{i \in \mathcal{S}_{\ell}} \int_{y = 0}^{\infty} \pi_i^{\ell}(y)(\mathcal{F}^{\ell}_i)\ud y + \sum_{\ell \in \{-,0\}} \sum_{i \in \mathcal{S}_{\ell}}p^{\ell}_i(\mathcal{F}^{\ell}_i) = 1,
	\end{align}
	where $\mathbb{K} := \mathbb{D}^{++}(0) + \Psi\mathbb{D}^{(-+)}(0)$ and $\alpha$ is a normalizing constant. 
\end{theo} 

\section{Discontinuous Galerkin Approximations}
	\label{sec:numframe}
%\subsection{Discontinuous Galerkin} 

Discontinuous Galerkin methods are used to approximate the solution to a system of partial differential equations. A brief description of these methods is as follows \cite{c99}. On the domain of the approximation, consider a finite sequence of so-called \emph{nodal points}. We refer to each interval between two consecutive nodal points as a \emph{mesh}, and the combination of meshes and nodal points as a \emph{stencil}. Within each mesh, we have a finite element approximation, which constructs a finite-dimensional smooth Sobolev space by choosing appropriate piecewise polynomial basis functions, and then projects the partial differential equations onto this space. This projection leads to a new system of equations, referred to as the \emph{weak form} of the original PDEs. 

There is a \emph{flux operator} moving probability from one mesh to another, in a manner similar to the underlying principle of a finite volume approximation: integrating the PDEs over each mesh and then constructing a new system of ordinary differential equations, which describe the change in the integral over the mesh. This method conserves probability, and can handle discontinuities, such as jumps and point masses. 

Discontinuous Galerkin methods lead to global approximations in the space of piecewise functions. Intuitively, we sacrifice the continuity between meshes to gain the conservation of probability. 
% The local approximations within a mesh can be as smooth as desired, by appropriate choices of the basis functions.

\subsection{Application to a Stochastic Fluid-Fluid Model}

Here, we construct a discontinuous Galerkin method to approximate the operator matrix $\mathbb{B}$ and subsequently the operator matrix $\Psi$, the two key ingredients of the joint stationary distribution for $\{X_t, Y_t, \varphi_t\}$. We begin with approximating the joint density $f_i(x,t)$ of $\{X_t, \varphi_t\}$: 
	\begin{align*} 
	    f_i(x,t) := \frac{\partial}{\partial x} 	\mathds{P}\left[X_t \leq x, \varphi_t = i\right],
	\end{align*} 
which satisfies the system of partial differential equations 
\begin{align}
	\label{eq:pde_density}
\frac{\partial}{\partial t} f_i(x,t) = \sum_{j\in \mathcal{S}}  f_j(x,t)T_{ji} - c_i \frac{\partial}{\partial x} f_i(x,t), 
\end{align}
subject to suitable boundary conditions~\cite{bo2014}.
% 
%\intrusion{State the suitable boundary conditions}
%We refer to the two terms on the right-hand side of~\eqref{eq:pde_density} as the \emph{stochastic component} and the \emph{drift component}, respectively. 

%The idea is to construct the weak form for these two components separately, and then combine them appropriately in order to construct an approximation to the operator $B$. 

While $X_t \in [0,\infty)$, any numerical approximation by necessity has to take place on a finite interval. Clearly, the state space truncation results in a point mass at the upper bound, which we have to address properly. It is important to choose an interval sufficiently large in order to control the error induced by the artificial upper bound for $X_t$.  We shall further comment on this in Section~\ref{sec:numexp}, where we report our numerical experiments. 

Let $[0,\mathcal{I}]$ be the domain of the approximation, where $\mathcal{I} < \infty$. We denote by $\left\{x_1, x_2, \ldots , x_K\right\}$ a finite sequence of $K$ nodal points on $[0,\mathcal{I}]$, with $x_1 := 0$ and $x_K := \mathcal{I}$, and by $\left\{\mathcal{D}_1, \ldots, \mathcal{D}_{K - 1}\right\}$ the sequence of corresponding meshes, $\mathcal{D}_i := [x_i, x_{i + 1}]$, for $i = 1, \ldots, K  - 1$. As there are two point masses, one at zero where there is a regulated boundary for $X_t$ and another at $\mathcal{I}$ where there is an artificial upper bound, the two end meshes, $\mathcal{D}_1$ and $\mathcal{D}_{K - 1}$, each of which contains both a point mass and density. 

For $k = 1, \ldots, K - 1$, we choose $N_k$ functions $\phi_n^k: \mathcal{D}_k \mapsto [0,\infty)$, $n =  1, \ldots, N_k$,  to be the basis functions, the span of which forms our approximation space, $V_K :=  \oplus_{k = 1}^{K - 1}\{\phi_1^k, \ldots, \phi_{N_k}^k\}$. Then, a function $u_i(\cdot, \cdot) \in V_K$ has the form: 
\begin{align}
	\label{eqn:uxt} 
	u_i(x,t) = \sum_{k = 1}^{K - 1}\sum_{n = 1}^{N_k} \alpha_{i,n}^k(t) \phi_n^k(x) \quad \mbox{for } x \in [0, \mathcal{I}] \mbox{ and } t \geq 0,
\end{align}
for some coefficient functions $\alpha^{k}_{i,n}(t)$. To construct an approximation for $f_i(x,t)$ we need to determine these functions $\alpha_{i,n}^{k}(t)$ or, equivalently, the $N$-dimensional row vector  
	\begin{align*} 
		\bs{\alpha}_i(t) := \left(\bs{\alpha}^1_i(t), \ldots, \bs{\alpha}_i^{K - 1}(t)\right),  
	\end{align*} 
	where  
	\begin{align*} 
		\bs{\alpha}_{i}^{k}(t) := \left(\alpha^k_{i, 1}(t), \ldots, \alpha^k_{i, N_k}(t)\right) \quad \mbox{for } k = 1, \ldots, K - 1,
	\end{align*} 
	and $N:= \sum\limits_{k = 1}^{K - 1} N_k $ is the total number of basis functions on $[0, \mathcal{I}]$. 
	
To that end, there are three important matrices we need to introduce. The first two matrices, $M$ and $G$, are $N \times N$ block-diagonal and detail the dynamics within each mesh:
	\begin{align*} 
		M & := \left[\begin{array}{cccc}  M^1 & & \\ 
		  & \ddots & \\\ 
		  &           & M^{K - 1}  
		  \end{array} \right], \quad G := \left[\begin{array}{cccc}  G^1 & & \\ 
		  & \ddots & \\\ 
		  &           & G^{K - 1}  
		  \end{array} \right],	\end{align*} 
	where, for $m, n \in \{1, \ldots, N_k\}$,
 \begin{align}
 	 [M^k]_{mn} & := \int_{\mathcal{D}_k} \phi_m^k(x) \phi_{n}^k(x) \ud x,  \\
	[G^k]_{mn} & := \int_{\mathcal{D}_k} \phi_m^k(x)\left[\frac{\partial}{\partial x}\phi^k_{n}(x)\right]  \ud x. 
	\end{align} 
	
	The third matrix, $F_i$, $i \in \mathcal{S}$, is related to the dynamics between adjacent meshes: it is the \emph{flux operator} moving probability from one mesh to another. Let $u^k_i(x,t)$ be the projection of $u_i(x,t)$ onto the mesh $\mathcal{D}_k$. A central idea of the discontinuous Galerkin method is that the values of $u^k_i(x,t)$ on different meshes are linked to each other only through a \emph{numerical flux}~$f^*$ and that one can consider the approximation 
	\begin{align*} 
		\left[u_i^k(x,t)\phi_m^k(x)\right]_{x_k}^{x_{k + 1}}  =  \left[u_i^k(x,t)\phi_m^k(x)\right]_{x_k^L}^{x_{k}^R} \approx \left[f^*_i(x,t)\phi_m^k(x)\right]_{x_k^L}^{x_{k}^R},
	\end{align*} 
	where $x_k^R$ and $x_k^L$ denote the right and left endpoints of the $k$th mesh, respectively, and $\left[g(x)\right]_{x_a}^{x_b} :=  g(x_{b}) - g(x_{a})$ for any function $g$.  
	
There are many options for $f^*$. Here, we choose a \emph{first-order up-winding scheme} \cite{b2003}, that is, 
	\begin{align} 
		\label{eqn:upw}
		f^*_i (x,t) := \eta(\mbox{sgn}(c_i),x) \lim_{\varepsilon \downarrow 0} u_i(x - \varepsilon c_i,t),  
	\end{align}
for $i \in \mathcal{S}$ and $x \in \{x_1^L, \ldots, x_K^L, x_1^R, \ldots, x_K^R\}$, where $\eta(\mbox{sgn}(c_i),x)$ is an adjustment parameter for when we have a stencil with meshes of different structures; note that the numerical flux $f^*_i(x,t)$ is defined at nodal points only. Suppose we are considering the $k$th mesh, $\mathcal{D}_k$. Then, we define the function $\eta(\mbox{sgn}(c_i),x)$ to be 
	\begin{align*} 
		\eta(\mbox{sgn}(c_i),x) := \left\{\begin{array}{ll} 
			\eta_{k + 1, k} & \mbox{if $x = x_k^{R}$ and $c_i < 0$}, \quad \mbox{ for } k = 1, \ldots, K - 2, \\ 
	\vspace*{-0.3cm} \\
			\eta_{k - 1, k} &  \mbox{if $x = x_k^{L}$ and $c_i > 0$}, \quad\mbox{ for } k = 2, \ldots, K - 1, \\
	\vspace*{-0.3cm} \\ 
			1 & \mbox{otherwise.} % \\
%	\vspace*{-0.3cm} \\		
%			\eta_{k + 1, k} & \mbox{if $x = x_k^R$ and $c_i < 0$}, \\ 
%	\vspace*{-0.3cm} \\		
%			 1 &  \mbox{if $x = x_k^R$ and $c_i > 0$}, \\		
		\end{array}\right.
	\end{align*}  
	Here, the function $\eta_{\ell,k}$, $\ell \in \{k-1,k + 1\}$, is the ratio of the integrals of the basis functions transferring the probability from the $\ell$th mesh to the $k$th mesh, where 
	\begin{align}
		\eta_{\ell, k} := \frac{\displaystyle\int_{\mathcal{D}_k} \phi_m^{k}(x) \ud x }{\displaystyle\int_{\mathcal{D}_{\ell}} \phi_n^{\ell}(x) \ud x }
	\end{align}  
	for one and then all $m \in \{1, \ldots, N_{k}\}$ and $n \in \{1, \ldots, N_{\ell}\}$. In other words, $\eta_{\ell, k}$ is not dependent on the particular choice of basis functions, $m$ and $n$, but on their meshes, $k$ and $\ell$, only.

This choice of $f^*$ requires the flux information only from the left if the fluid rate $c_{\varphi_t}$ of $X_t$ is positive, and only from the right if that rate is negative. Intuitively, when $c_{\varphi_t} > 0$, the numerical flux of probability going from $\mathcal{D}_{k - 1}$ into $\mathcal{D}_k$ is the probability density accumulated on the right-hand edge of the approximation on $\mathcal{D}_{k - 1}$. Similarly, when $c_{\varphi_t} < 0$, the numerical flux going from $\mathcal{D}_{k}$ into $\mathcal{D}_{k - 1}$ is the probability density accumulated on the left-hand edge of the approximation on $\mathcal{D}_k$.

Thus, on nodal points we have 
	\begin{align} 
		\label{eqn:upw2}
		f_i^*\left(x,t\right) &  := \left\{\begin{array}{lr}
		 u_i\left(x_k^{L+},t\right) & \mbox{if } x = x_k^L \mbox{ and } c_i < 0, \\ 
		  \vspace*{-0.2cm} \\ 
		\eta_{k + 1, k}u_i\left(x_k^{R+},t\right) & \mbox{if } x = x_k^R \mbox{ and } c_i < 0, \\
		 \vspace*{-0.2cm} \\ 
		 \eta_{k - 1,k} u_i\left(x_{k }^{L-},t\right) & \mbox{ if } x = x_k^L \mbox{ and } c_i > 0,  \\
		 \vspace*{-0.2cm} \\ 
		  u_i\left(x_{k }^{R-},t\right) & \mbox{ if } x = x_k^R \mbox{ and } c_i > 0,  \\		 
		\end{array} \right.
	\end{align} 
	where for any function $g$ 
	 	\begin{align} 
			g(x^+) := \lim\limits_{\varepsilon \rightarrow 0} g(x + \varepsilon), \quad g(x^-) := \lim\limits_{\varepsilon \rightarrow 0} g(x - \varepsilon), 
		\end{align} 
		and is used to allow us access to the density on the left-hand edge of $\mathcal{D}_{k + 1}$ and the right-hand edge of $\mathcal{D}_{k - 1}$, respectively. 
	% 
%	\begin{rem} 
%	% 
%	Comparing the definition \eqref{eqn:upw} to \eqref{eqn:upw2} and \eqref{eqn:upw3}, observe that 
%		% 
%		\begin{align*} 
%			\lim_{\varepsilon \downarrow 0} u_i\left(x_k^L - \varepsilon c_i,t\right) = 
%			\left\{ \begin{array}{ll}  
%			u_i\left(x_k^L,t\right) &  \mbox{for } c_i < 0, \\
%			\vspace*{-0.3cm} \\
%			u_i\left(x_{k - 1}^R,t\right) & \mbox{for } c_i > 0, 
%			\end{array}
%			\right.
%		\end{align*} 
%		% 
%		and
%		%  
%		\begin{align*} 
%			\lim_{\varepsilon \downarrow 0} u_i\left(x_k^R - \varepsilon c_i,t\right) = 
%			\left\{ \begin{array}{ll}  
%			u_i\left(x_{k + 1}^L,t\right) &  \mbox{for } c_i < 0, \\
%			\vspace*{-0.3cm} \\
%			u_i\left(x_{k}^R,t\right) & \mbox{for } c_i > 0. 
%			\end{array}
%			\right.
%		\end{align*} 	
%	% 
%%	While $x_{k}^L = x_{k - 1}^R$, the two functions $u_i\left(x_k^L,t\right)$ and $u_i\left(x_{k - 1}^R, t\right)$ are not necessarily equal. This is because every basis function $\phi_t^k$ is defined to be zero outside of the $k$th mesh, so $\phi_t^k\left(x_{k - 1}^R\right) = \phi_t^k\left(x_{k + 1}^L\right) = 0$, while $\phi_t^k\left(x_k^L\right)$ and $\phi_t^k\left(x_k^R\right)$ are not necessarily zero.
%	\end{rem} 
	
	We are now ready to introduce the block-tridiagonal matrix $F_i$, 
		\begin{align*} 
			  F_i & := \left[\begin{array}{ccccccc} 
		  F^{11}_i & F_i^{12} &  & & \\ 
		  F^{21}_i & F_i^{22} & F_i^{23} &  \\
		             & \ddots & \ddots & \ddots   \\
		             		            & &  \ddots & \ddots & \ddots   \\
				            \vspace*{-0.2cm} \\
		             &        &    &  F^{K - 2, K - 3}_i   &  F^{K - 2, K - 2}_i & F^{K - 2, K -1}_i \\
		             &         &    &                            &  F^{K - 1, K - 2}_i & F^{K -1, K -1}_i 
		    \end{array}\right],
	\end{align*} 
	where each sub-block $F_{i}^{\ell k}$ is of dimension $N_\ell \times N_k$ and thus  
		\begin{align}
			\bs{\alpha}_i(t) F_i   &\;= \left(\sum_{j = 1}^{K - 1} \bs{\alpha}_i^j(t) F_i^{j1}, \ldots,  \sum_{j = 1}^{K - 1} \bs{\alpha}_i^j(t) F_i^{j, K - 1} \right)  \nonumber
		\end{align} 
	   is an $N$-dimensional row vector. Let $\bs{\phi}^{\ell}(x) := (\phi_1^{\ell}(x), \ldots, \phi^{\ell}_{N_{\ell}}(x))$ be the vector containing all basis functions on the $\ell$th mesh. We define for $c_i > 0$
	\begin{align*}
		F^{k - 1, k}_i
%		& = 
%		\eta_{k - 1, k} \left[\begin{array}{c} 
%		\phi_1^{k - 1}\left(x_{k}^{-}\right) \\
%		\vdots \\
%		\phi_{N_{k -1}}^{k - 1}\left(x_{k }^-\right) 
%		\end{array} \right] 
%		% 
%		\left[\begin{array}{ccc}
%			\phi_1^{k}\left(x_{k}\right) & \ldots & \phi_{N_{k}}^{k}\left(x_{k}\right) 
%		\end{array}\right] & \mbox{for } {k = 2, \ldots, K - 1}, \\
%		\vspace{-0.1cm}  \\
		& := \eta_{k - 1, k} \left[\bs{\phi}^{k - 1}(x_k^-)\right]^{\top} \bs{\phi}^k(x_k) \quad & \mbox{for } {k = 2, \ldots, K - 1}, \\ 
		F^{kk}_i 
%		& = 
%		{- \left[\begin{array}{c} 
%		\phi_1^{k}\left(x_{k+1}^{-}\right) \\
%		\vdots \\
%		\phi_{N_k}^{k}(x_{k+1}^{-}) 
%		\end{array} \right] \left[\begin{array}{ccc}
%			\phi_1^{k}\left(x_{k + 1}\right) & \ldots & \phi_{N_{k}}^{k}\left(x_{k + 1}\right) 
%		\end{array}\right]}
%		& \mbox{for } k = 1, \ldots, K - 1, \\
		& := - \left[\bs{\phi}^{k}(x_{k + 1}^-)\right]^{\top} \bs{\phi}^k(x_{k + 1}) \quad & \mbox{for } k = 1, \ldots, K - 1,  \\ 
		F_i^{k \ell} & := 0 \quad & \mbox{otherwise,}  
		%\end{align*} 
		% 
		\intertext{and for $c_i < 0$} 
		% 
		%\begin{align*} 
		F_{i}^{k + 1, k} 
%		& = -\eta_{k + 1, k} \left[\begin{array}{c} 
%			\phi_{1}^{k + 1}\left({x_{k + 1}^+}\right) \\
%			\vdots \\
%			\phi^{k + 1}_{N_{k + 1}}\left({x_{k + 1}^+}\right)
%		\end{array} \right] 			
%			%
%			\left[\begin{array}{ccc}
%			\phi_{1}^{k}\left({x_{k + 1}}\right) & \cdots & \phi_{N_k}^{k}\left({x_{k + 1}}\right) 
%			\end{array}\right] & \mbox{ for } k = 1, \ldots, K - 2, \\
	    % 
	    	& = -\eta_{k + 1, k} \left[\bs{\phi}^{k + 1}(x_{k + 1}^+)\right]^{\top} \bs{\phi}^k(x_{k + 1}) \quad & \mbox{for } k = 1, \ldots, K - 2,  \\ 
	    F^{kk}_i 
%	    & = \left[\begin{array}{c} 
%	    \phi_1^k\left({x_{k }^+}\right) \\
%	          \vdots \\
%	     \phi_{N_k}^{k}\left({x_{k }^+}\right)  
%	    \end{array}\right] \left[\begin{array}{ccc}
%	    	\phi_1^k\left({x_{k }}\right) & \cdots & \phi_{N_k}^{k}\left({x_{k }}\right) 
%	    \end{array}\right] & \mbox{ for } k = 1, \ldots, K - 1,  \\
	    % 
	    & = \left[\bs{\phi}^{k}(x_{k}^+)\right]^{\top} \bs{\phi}^k(x_{k}) \quad & \mbox{for } k = 1, \ldots, K - 1,  \\ 
		F_i^{k\ell} & = 0 \quad & \mbox{otherwise.} 
	\end{align*} 
	\begin{lem}
	For $k = 1, \ldots, K - 1$ and $m = 1, \ldots, N_k$, 
	%we define the $m$th element of the $k$th component of  $\bs{\alpha}_i(t) F_i$ to be 
		\begin{align} 
		% 
		%[\bs{\alpha}_i(t) F_i  ]_{km} & \;=
		 \left[\sum_{j = 1}^{K - 1} \bs{\alpha}_i^j(t) F_i^{jk}\right]_m % \nonumber \\
		 = 	%\left\{ 
%			\begin{array}{ll} 
				-\left[f^*_i(x,t)\phi_m^k(x)\right]_{x_k^L}^{x_{k}^R}.  
				% & \mbox{ if } c_i < 0, \\ 
%								\vspace*{-0.2cm} \\	
%				-\left[f^*_i(x,t)\phi_m^k(x)\right]_{x_k^L}^{x_{k}^R}  & \mbox{ if } c_i > 0. 
%			\end{array}\right.
\label{eqn:lemma} 
		\end{align} 
		% 
		%where $k$ denotes the $k$th sub-block of the vector $\sum\limits_{\ell = 1}^{K - 1} \bs{\alpha}_i^{\ell}(t) F^{\ell k}_i$, and $m$ denotes the $m$th element of this sub-block. In other words, this is the $(\sum_{i = 1}^{k- 1} N_{i} + m)$th element of the row vector $\bs{\alpha}_i(t) F_i  $.  
	\end{lem} 
	\begin{proof} 
		We begin with the RHS of~\eqref{eqn:lemma}. For $c_i > 0$, 
		\begin{align} 
	\left[f^*_i(x,t)\phi_m^k(x)\right]_{x_k^L}^{x_{k}^R} 
	& = f_i^*\left(x_k^R, t\right) \phi_m^k\left(x_{k}^R\right) - f_i^*\left(x_k^L, t\right) \phi_m^k\left(x_k^L\right) \nonumber \\ 
	& = u_i\left(x_{k}^{R-}, t\right) \phi_m^k\left(x_k^R\right) - \eta_{k - 1, k} u_i\left(x_k^{L-},t\right) \phi_m^k(x^L_k) \nonumber \\
	& = u_i\left(x_{k + 1}^-, t\right) \phi_m^k\left(x_{k + 1}\right) - \eta_{k - 1, k} u_i\left(x_k^-,t\right) \phi_m^k(x_k). 
	\label{eqn:fstar3} 
		\end{align} 
		As each basis function $\phi_n^j$, for $n = 1, \ldots, N_j$, is trivially zero outside its $j$th mesh, we have 
		\begin{align}
		u_i\left(x_{k + 1}^-, t\right) & = \sum_{j = 1}^{K - 1}\sum_{n = 1}^{N_j} \alpha_{i, n}^{j}(t) \phi_n^{j}\left(x_{k + 1}^{-}\right) = \sum_{n = 1}^{N_k} \alpha_{i, n}^{k}(t) \phi_n^{k}\left(x_{k + 1}^{-}\right), \label{eqn: ui1} \\
		u_i\left(x_k^-,t\right) & =  \sum_{j = 1}^{K - 1}\sum_{n = 1}^{N_j} \alpha_{i, n}^{j}(t) \phi_n^{j}\left(x_{k}^{-}\right) = \sum_{n = 1}^{N_{k - 1}} \alpha_{i, n}^{k - 1}(t) \phi_n^{k - 1}\left(x_{k}^{-}\right).
		\label{eqn: ui2}
		\end{align} 
		Substituting \eqref{eqn: ui1} and \eqref{eqn: ui2} into \eqref{eqn:fstar3} leads to 
		\begin{align} 
	& \left[f^*_i(x,t)\phi_m^k(x)\right]_{x_k^L}^{x_{k}^R} \nonumber \\
	& \quad =  \sum_{n = 1}^{N_k} \alpha_{i, n}^{k}(t) \phi_n^{k}\left(x_{k + 1}^{-}\right)\phi_m^k(x_{k +1}) - \eta_{k  - 1, k}\sum_{n = 1}^{N_{k - 1}} \alpha_{i, n}^{k - 1}(t) \phi_n^{k - 1}\left(x_{k}^{-}\right)\phi_m^k(x_k) \nonumber \\
	& \quad = - \left[\sum_{j = 1}^{K - 1} \bs{\alpha}_i^j(t) F_i^{jk}\right]_m. \nonumber
		\end{align} 
		The argument for $c_i < 0$ follows analogously. 
	\end{proof} 
	\begin{theo} 
		The weak formulation of the PDEs \eqref{eq:pde_density} is the following system of ordinary differential equations 
\begin{align} 
	\label{eqn:wf}
	\frac{\ud}{\ud t} \boldsymbol{\alpha}_{i}(t) = \sum_{j \in \mathcal{S}}  \boldsymbol{\alpha}_{j}(t) T_{ji} + c_i \boldsymbol{\alpha}_{i}(t)  (G + F_i)M^{-1} \quad \mbox{ for } i \in \mathcal{S}. 
\end{align}
\end{theo}
\begin{proof} 
	Consider Equation~\eqref{eq:pde_density}, for $i \in \mathcal{S}, t \geq 0, x \in [0,\mathcal{I}]$, which we restate below: 
	\begin{align}
 	\label{eq:pde11}
\frac{\partial}{\partial t} f_i(x,t) = \sum_{j\in \mathcal{S}}  f_j(x,t)T_{ji} - c_i \frac{\partial}{\partial x} f_i(x,t).  
\end{align}
	   For details on the steps of discontinuous Galerkin methods, see \cite{c99}. Here, we start by replacing the density function $f_i(x,t)$ in \eqref{eq:pde11} by its approximation $u_i(x,t)$ to obtain 
	\begin{align} 
		\label{eqn:pdek} 
	\frac{\partial}{\partial t} u_i(x,t) - \sum_{j\in \mathcal{S}} u_j(x,t) T_{ji}  + c_i \frac{\partial}{\partial x} u_i(x,t) = 0.
	\end{align} 
	 
	Multiplying both sides of \eqref{eqn:pdek} by a basis function $\phi_{m}^k, m \in \{1, \ldots, N_k\}$ and $k \in \{1, \ldots, K = 1\}$, and then integrating over the approximation domain $[0,\mathcal{I}]$ gives
	\begin{align} 
		%\label{eqn:testint} 
		\int_{[0,\mathcal{I}]} \left[\frac{\partial}{\partial t} u_i(x,t)   -  \sum_{j\in \mathcal{S}}  u_j(x,t) T_{ji} 
		+  c_i \frac{\partial}{\partial x} u_i(x,t)  \right]\phi_m^k(x) \ud x & = 0,  \nonumber 
	\intertext{which reduces to}
		\label{eqn:testint} 
		\int_{\mathcal{D}_k} \left[\frac{\partial}{\partial t} u_i^k(x,t)   -  \sum_{j\in \mathcal{S}}  u_j^k(x,t) T_{ji} 
		+  c_i \frac{\partial}{\partial x} u_i^k(x,t)  \right]\phi_m^k(x) \ud x & = 0. 
	\end{align} 
	
	Expanding and then integrating the third term by parts leads to 
	\begin{align*} 
		%\label{eqn:ibp}
	    c_i  \left[u^k_i(x,t) \phi_m^k(x)\right]^{x_{k}^R}_{x_{k}^L} - \int_{\mathcal{D}_k}c_i u_i^k(x,t) \frac{\ud}{\ud x}\phi_m^k(x)\ud x, 
	\end{align*} 
where the first part can be approximated by using the numerical flux $f^*_i(x,t)$, that is,  
	\begin{align}
		\label{eqn:approx_cu} 
		c_i \left[u_i^k(x,t) \phi_m^k(x)\right]^{x_k^R}_{x_k^L} & 
		\approx %  {\left\{\begin{array}{ll} 
				c_i\left[f^*_i(x,t)\phi_m^k(x)\right]_{x_k^L}^{x_{k}^R}  
				% & \mbox{ if } c_i > 0, \\ 
				% \vspace*{-0.2cm} \\
%				c_i\left[f^*_i(x,t)\phi_m^k(x)\right]_{x_k^L}^{x_{k}^R}  & \mbox{ if } c_i < 0, \\ 
%			\end{array}\right.}  \nonumber \\
		 =:  -c_i  [\bs{\alpha}_i(t) F_i ]_{km}.
	\end{align}
%	% 
%	Recall that the subscript $(k,m)$ refers to the $m$th element of the $k$ sub-block of the row vector $\bs{\alpha}_i(t)F_i$.  
	
%	\intrusion{you are here!} 
	
	Substituting \eqref{eqn:approx_cu} into \eqref{eqn:testint} and expanding $u^k_i(x,t)$ gives
	\begin{align*} 
		\int_{\mathcal{D}_k} 	\frac{\partial}{\partial t} \left[\sum_{n = 1}^{N_k} \alpha_{i,n}^{k}(t)\phi_n^k(x)\right] \phi_m^k(x)\ud x  - \int_{\mathcal{D}_k} \sum_{j\in \mathcal{S}}  \left[\sum_{n = 1}^{N_k} \alpha_{j,n}^{k}(t)T_{ji} \phi_n^k(x)\right] \phi_m^k(x) \ud x  \nonumber \\
		   - c_i [\bs{\alpha}_i(t) F_i ]_{km} - \int_{\mathcal{D}_k}c_i \left[\sum_{n = 1}^{N_k} \alpha_{i,n}^{k}(t) \phi_n^k(x) \right] \frac{\partial}{\partial x}\phi_m^k(x)\ud x  = 0. 
	\end{align*}
		
	Finally, we switch the order of integrals and summations to arrive at  
	\begin{align*} 
		\sum_{n = 1}^{N_k} \left[ \frac{\ud}{\ud t}  \alpha_{i,n}^{k}(t) \int_{\mathcal{D}_k} \phi_n^k(x) \phi_m^k(x)\ud x  -  \sum_{j\in \mathcal{S}} \alpha_{j,n}^{k}(t)T_{ji}   \int_{\mathcal{D}_k}  \phi_n^k(x) \phi_m^k(x) \ud x \right. \\
		\left.  - c_i \alpha_{i,n}^{k}(t) \int_{\mathcal{D}_k}  \phi_n^k(x)  \frac{\partial}{\partial x}\phi_m^k(x)\ud x \right]  - c_i [\bs{\alpha}_i(t) F_i]_{km}  = 0.  
	\end{align*}
	Equivalently, this can be written in matrix form as 
	\begin{align*}
		\frac{\ud}{\ud t} \bs{\alpha}_i(t)M  = \sum_{j \in \mathcal{S}} \boldsymbol{\alpha}_j(t)T_{ji}M  + c_i\bs{\alpha}_i(t)(G + F_i), 
	\end{align*}  
		which completes the proof. 	
%
%		\intrusion{Something is needed at the beginning of this proof, which currently only spells out the steps to obtain the theorem but does not explain why we are following these steps.}
\end{proof} 

\begin{rem} 
	\label{rem:conserv}
	Let $V_K$ be the approximation space as defined in~\eqref{eqn:uxt}. 
	If
	\begin{enumerate} 
	\item the weights $\bs{\alpha}_i(t)$ satisfy~\eqref{eqn:wf}, 
	\item the eigenvalues of $c_i(G + F_i)M^{-1}$ are in the negative real half of the complex plane including zero, and
	\item only one basis function is non-zero on each boundary of each mesh,
	\end{enumerate} 
	 then 
	 \begin{align*}
	 	\int_{[0,\mathcal{I}]}u_i(x,t) \ud x = \int_{[0,\mathcal{I}]}u_i(x,0) \ud x.  
	 \end{align*}

Intuitively, the second and third conditions are to ensure stability in the solution and conservation of the transfer of probability across meshes, respectively. 

%\intrusion{\hp{I am not too sure what needs to be cited, to have a well behaved approximation we are request the generator sum to zero (which gives conservation of mass) and then second is that we want negative eigenvalues such that
%the sub blocks do not explode under iterations. And only one basis passing probability is a well-possed. All in all, it simply says make sure you choose sensible subspace to project your problem. }} 
%%
%\intrusion{Can we claim that Lemma 3.2 is a (weaker) result of the DG methods being locally conservative?} 
\end{rem} 
  
Defining the discontinuous Galerkin infinitesimal operator 
	\begin{align} 
		\mathcal{Q}^i := c_i(G + F_i)M^{-1} \quad \mbox{ for } i \in \mathcal{S},
	\end{align} 
	we construct a DG approximation $\mathcal{B}$ for the operator matrix $\mathbb{B}$ as follows. 
%	Recall that in the vector $\bs{\alpha}_i(t)$ we arrange the elements $\alpha_{i,n}^k(t)$ first according to the meshes, and then within each mesh according to a pre-defined order of basis functions for that mesh: 
%	% 
%		\begin{align*} 
%			\bs{\alpha}_i(t) = \vligne{\alpha_{i,1}^{1}(t) &  \cdots &  \alpha_{i,N_1}^{1}(t) &  \cdots &  \alpha_{i, 1}^{K - 1}(t) &  \cdots &  \alpha_{i, N_{K - 1}}^{K - 1}(t) }.
%		\end{align*} 	
	For $i \in \mathcal{S}$ and $\ell \in \{+, -, 0\}$, define {$\gamma_{i}^{ \ell}$} to be an index set: given a region $\mathcal{F}_{i}^{\ell}$ of $X_t$ (defined in~(\ref{eqn:Fi1}--\ref{eqn:Fi3})), {$\gamma_{i}^{\ell}$} is the set of meshes included in $\mathcal{F}^{\ell}_{i}$. For example, if $\mathcal{F}_{1}^{+} = \mathcal{D}_{1} \cup \mathcal{D}_2 \cup \mathcal{D}_5$, then ${\gamma_{1}^{+}} = \{1,2,5\}$. Note that the nodal points should be chosen such that {${\gamma_{i}^{\ell}}\cap \gamma_{i}^{m} = \varnothing$ for all $\ell, m \in \{+,-,0\}, \ell \neq m$, and $\gamma_{i}^{+} \cup \gamma_{i}^{-} \cup \gamma_{i}^{0} = \{1, 2, \ldots, K - 1\}$.}
%	
%	the vector of basis functions whose support is inside $\mathcal{F}_{1}^{+}$ is 
%	% 
%		\begin{align*} 
%			\vligne{\alpha_{1,1}^{1} & \alpha_{1,2}^{1} & \cdots & \alpha_{1,N_1}^{1} & \alpha_{1,1}^{2} & \alpha_{1,2}^{2} & \cdots & \alpha_{1,N_2}^{2}},  
%		\end{align*} 
%	% 
%	and 
%		\begin{align} 
%			\label{eqn:exg}
%			\gamma_{1+} = \vligne{(1,1) & \cdots & (1, N_1) & (2, 1) \cdots & (2, N_2)}.
%		\end{align} 
		%  
%	Furthermore, we denote 
%		\begin{itemize} 
%		\item[-] by $\gamma_{i \ell}(R)$ the right-most index of $\gamma_{i\ell}$ (which in~\eqref{eqn:exg} would be $(2, N_2)$), 
%		\item[-] by $\gamma_{i \ell}(L)$ the left-most index of $\gamma_{i\ell}$ (correspondingly, $(1,1)$), 
%		\item[-] by $\gamma_{i \ell}(R - n)$ the $n$th basis function from the right-most, and 
%		\item[-] by $\gamma_{i \ell}(L + n)$ the $n$th basis function from the left-most. 
%		\end{itemize} 
	
	We choose each approximation matrix $\mathcal{B}_{ij}^{\ell m}$ to be $N\times N$. Similar to the definition for the operator matrix $\mathbb{B}$, introduced in Section \ref{subsec:B_operators}, there are three cases. 
	% 
%	\intrusion{Is this supposed to be the basis function $n$th furthest \emph{from the right}?}
	
	\begin{enumerate}
		\item[\textbf{Case 1.}] When $i \neq j$, each $N_k \times N_k$ sub-block $\left[\mathcal{B}_{ij}^{\ell m}\right]_{kk}$ is given by 
%		 all elements of $\mathcal{B}_{ij}^{\ell m}$ are inherently zero except for 	one sub-block $[\mathcal{B}_{ij}^{\ell m}]_{\gamma_{i\ell} \times \gamma_{jm}}$ which is given by 
%	%
%	\begin{align*} 
%		{[\mathcal{B}_{ij}^{\ell m}]_{\gamma_{i\ell} \times \gamma_{jm}}  := T_{ij} I_{\gamma_{i \ell} \times \gamma_{j m}}.}
%	\end{align*} 
		\begin{align*} 
			\left[\mathcal{B}_{ij}^{\ell m}\right]_{kk} := T_{ij}I_{N_k}\mathds{1}_{\left\{k \in {\gamma_{i}^{\ell}} \cap {\gamma_{j}^{m}}\right\}} \quad \mbox{for } k = 1, \ldots, K - 1.
		\end{align*} 
	
		\item[\textbf{Case 2.}] When $i = j$ and $\ell \neq m$, 
%all elements of the matrix $\mathcal{B}_{ij}^{\ell m}$ are inherently zero except for
%	% 
%	\begin{align*} 
%		\left[\mathcal{B}_{ii}^{\ell m}\right]_{\gamma_{i\ell(L)} \times \gamma_{i,m}(R)\ldots\gamma_{i,m}(R-k)} & := \vligne{\mathcal{Q}^i_{\gamma_{i \ell}(L),  \gamma_{i m}(R)} & \cdots &  \mathcal{Q}^i_{\gamma_{i \ell}(L), \gamma_{i m} (R-k)}}   & \mbox{ for } c_i < 0, \\
%	    \left[\mathcal{B}_{ii}^{\ell m}\right]_{\gamma_{i\ell(R)} \times \gamma_{i,m}(L)\ldots\gamma_{i,m}(L+k)}
%	    & := \vligne{\mathcal{Q}^i_{\gamma_{i \ell}(R),\gamma_{i m} (L) }  & \cdots  & \mathcal{Q}^i_{\gamma_{i \ell}(R), \gamma_{i m} (L+k)}}   & \mbox{ for } c_i > 0.
%	\end{align*} 
%	% 
%	where $k$ is the number of basis functions in the adjacent mesh in the direction of the flux. 
%	\intrusion{(a) For the first definition, does this mean do we reverse the order of the basis functions? (b) What does ``adjacent'' mean here? According to the example, it looks like $k$ is the number of basis functions of the mesh that has the first (on that mesh) basis function $\gamma_{i m}(L)$.} 
	\begin{align*}
		\left[\mathcal{B}_{ii}^{\ell m}\right]_{k, k + 1} :=  \mathcal{Q}^i_{k, k + 1} \mathds{1}_{\left\{k \in {\gamma_{i}^{\ell}}, k + 1 \in {\gamma_{j}^{m}}\right\}} \quad \mbox{for } c_i > 0 \mbox{ and } k = 1, \ldots, K - 2, \\
		\left[\mathcal{B}_{ii}^{\ell m}\right]_{k, k - 1} :=  \mathcal{Q}^i_{k, k - 1} \mathds{1}_{\left\{k \in {\gamma_{i}^{\ell}}, k - 1 \in {\gamma_{j}^{m}}\right\}} \quad \mbox{for } c_i < 0 \mbox{ and } k = 2, \ldots, K - 1.	
	\end{align*}

		\item[\textbf{Case 3.}] When $i = j$ and $\ell = m$, 
%	all elements of $B_{ij}^{\ell m}$ are inherently zero except for
%	% 
%	\begin{align*}
%		\left[\mathcal{B}_{ii}^{\ell \ell}\right]_{\gamma_{i\ell} \times \gamma_{i\ell}} := \left\{\begin{array}{rl} 
%			T_{ii} I_{\gamma_{i \ell} \times \gamma_{i \ell}} + \mathcal{Q}_{\gamma_{i \ell} \times \gamma_{i \ell}}^{i} & \mbox{ for } c_i < 0, \\
%			\vspace*{-0.3cm} \\
%			T_{ii} I_{ \gamma_{i \ell} \times \gamma_{i \ell}} + \mathcal{Q}_{\gamma_{i \ell} \times \gamma_{i \ell}}^i  & \mbox{ for } c_i > 0.		
%			\end{array}\right.  
%	\end{align*} 
%	\intrusion{Vikram, originally you had it has $\mathcal{Q}^+$ and $\mathcal{Q}^{-}$, necessitating two different definitions. However, I think it's meant to be $i$ instead of $+$ or $-$, which, as Nigel pointed out, means both definitions are identical. Is this correct?} 
%	% 
	\begin{align*}
		\left[\mathcal{B}_{ii}^{\ell \ell}\right]_{kk} := \left(T_{ii}I_{N_k} + \mathcal{Q}^i_{kk}\right)\mathds{1}_{\{k \in {\gamma_{i}^{\ell}}\}} \quad \mbox{for } k = 1, \ldots, K - 1.
	\end{align*} 
	\end{enumerate} 	 
	
Next, we approximate the operators $\mathbb{R}_i^{\ell}$ by an $N \times N$ block-diagonal matrix $\mathcal{R}_i^{\ell}$, with diagonal sub-blocks $[\mathcal{R}_i^{\ell}]_{kk}$ given by 
	\begin{align}
		\label{eqn:DG-R} 
		\left[\mathcal{R}_i^{\ell}\right]_{kk} := \diag \left[\left(\frac{1}{|\rho_{i,(k,n)}|}\right)_{n  =1, \ldots,N_k}\right]\mathds{1}_{\left\{k \in {\gamma_i^{\ell}}\right\}},
	\end{align}  	
	where $\rho_{i,(k,n)} \neq 0$ is an approximation of the fluid rate $r_i(X_t)$ of $Y_t$, given $\varphi_t = i$, $X_t \in \mathcal{D}_k$, and the basis function $\phi_n^k$. These functions $\rho_{i,(k,n)}$ can sensibly be chosen to be: 
	\begin{align*} 
		\rho_{i,(k,n)} := \int_{\mathcal{D}_k} r_i(x) \phi_{n}^k(x) \ud x.
	\end{align*} 
	% Note to ourselves: in Vikram's code there is no \phi_n(x)

	{Define $\mathcal{B}^{\ell m} := [\mathcal{B}_{ij}^{\ell m}]_{i \in \mathcal{S}_{\ell}, j \in \mathcal{S}_m}$ for $\ell, m \in \{+, -, 0\}$, and $\mathcal{R}^{\ell} := \diag(\mathcal{R}^{\ell}_i)_{i \in \mathcal{S}_{\ell}}$ for $\ell \in \{+, -\}$.} Putting things together, a DG approximation of the operator $\mathbb{D}_{ij}^{\ell m}$ is given by the $N \times N$ matrix 
		\begin{align*}
			\mathcal{D}_{ij}^{\ell m}(s) := 
			\left[\mathcal{R}^{\ell}\left(\mathcal{B}^{\ell m} - s I + \mathcal{B}^{\ell 0} (s I - \mathcal{B}^{00})^{-1} \mathcal{B}^{0m}\right)\right]_{ij}, 
		\end{align*} 
	 for $s \in \mathds{C}$ with Re$(s) > 0$ and for $\ell, m \in \{+, -\}$. 
	 
	 Let $\mathcal{D}^{\ell m}(s) := [\mathcal{D}_{ij}^{\ell m}(s)]_{i \in \mathcal{S}_{\ell}, j \in \mathcal{S}_{m}}$, then a DG approximation $\uppsi$ of the operator $\Psi$, the latter describing the probability of the fluid $Y_t$ starting at level zero and returning there for the first time, is an $(N \times |\mathcal{S}_{+}|) \times (N \times |\mathcal{S}_{-}|)$ matrix solution to 
%	 given by 
%	 % 
%	 \begin{align}
%	 	\label{eqn:approxPsi}
%	 	\uppsi(s) := \int_0^{\infty} \ue^{\mathcal{D}^{++}(s)y}[\mathcal{D}^{+-}(s) + \uppsi(s)\mathcal{D}^{-+}(s)\uppsi(s)]\ue^{\mathcal{D}^{--}(s)y}\ud y.
%	 \end{align} 
%     %  
%	 Adapting the proof of Theorem~\ref{theo:Psi} in \cite{bo2014}, we can show that for Re$(s) > 0$ the integral~\eqref{eqn:approxPsi} exists and satisfies
	  the equation 
	 \begin{align}
	 	\label{eqn:RiccatiPsi}
	 	\mathcal{D}^{+-}(s) + \uppsi(s)\mathcal{D}^{-+}(s)\uppsi(s) + \mathcal{D}^{++}(s)\uppsi(s) + \uppsi(s)\mathcal{D}^{--}(s) = 0. 
	 \end{align} 
	 The matrix equation~\eqref{eqn:RiccatiPsi} can be solved using one of the many algorithms suggested by Bean \emph{et al.}~\cite{bot08}. We then use this approximation $\uppsi$ to evaluate the limiting density, according to Theorem~\ref{theo:density}.
	 
	 \subsection{An Example} 

Consider a process $\{X_t, Y_t, \varphi_t\}$ where there is only one phase $i = 1$ in $\mathcal{S}$, in other words, there is no Markov modulation. Consequently, the PDE \eqref{eq:pde_density} has the drift component only: 
		\begin{align*} 
			\frac{\partial}{\partial t} f(x,t)  = -c \frac{\partial}{\partial x} f(x,t) \quad \mbox{ for } t \geq 0 \mbox{ and } x \in [0, \mathcal{I}].  
		\end{align*} 
	Suppose we would like to develop a DG approximation for the density function $f(x,t)$ over an interval $[0, 2.75]$. Consider the four meshes 
	\begin{align*} 
	   \mathcal{D}_1:= [0,0.25], \mathcal{D}_2 := [0.25, 1.25], \mathcal{D}_3 := [1.25, 2.25], \mbox{ and } \mathcal{D}_4 := [2.25, 2.75].
	 \end{align*} 
	 We choose the following basis functions:    
	\begin{align}
		\phi_1^1(x) & := 1 \quad \mbox{for } x \in \mathcal{D}_1 ,  \\
		\phi_1^2(x) & := -x + 1.25, \;\; \phi_2^2(x) := x - 0.25 \quad \mbox{for } x \in \mathcal{D}_2, \\ 
		\phi_1^3(x) & := -x + 2.25, \;\; \phi_2^3(x) := x - 1.25 \quad \mbox{for } x \in \mathcal{D}_3, \\ 
		\phi_1^4(x) & := 1 \quad \mbox{for } x \in \mathcal{D}_4, 
	\end{align} 
	as depicted in Figure~\ref{fig:example}. 

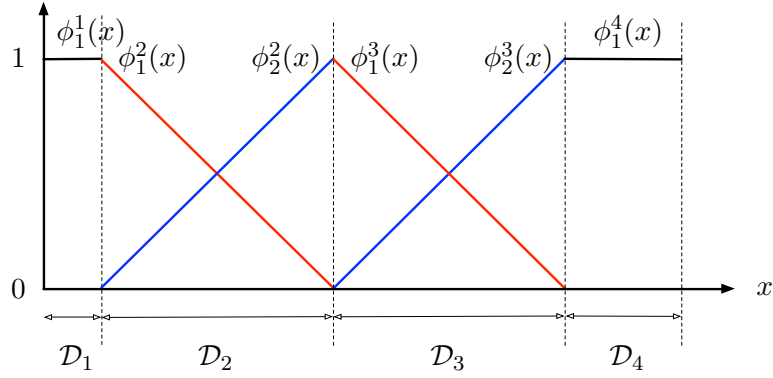
\begin{figure}[!h]
\centering
\includegraphics[scale=0.3]{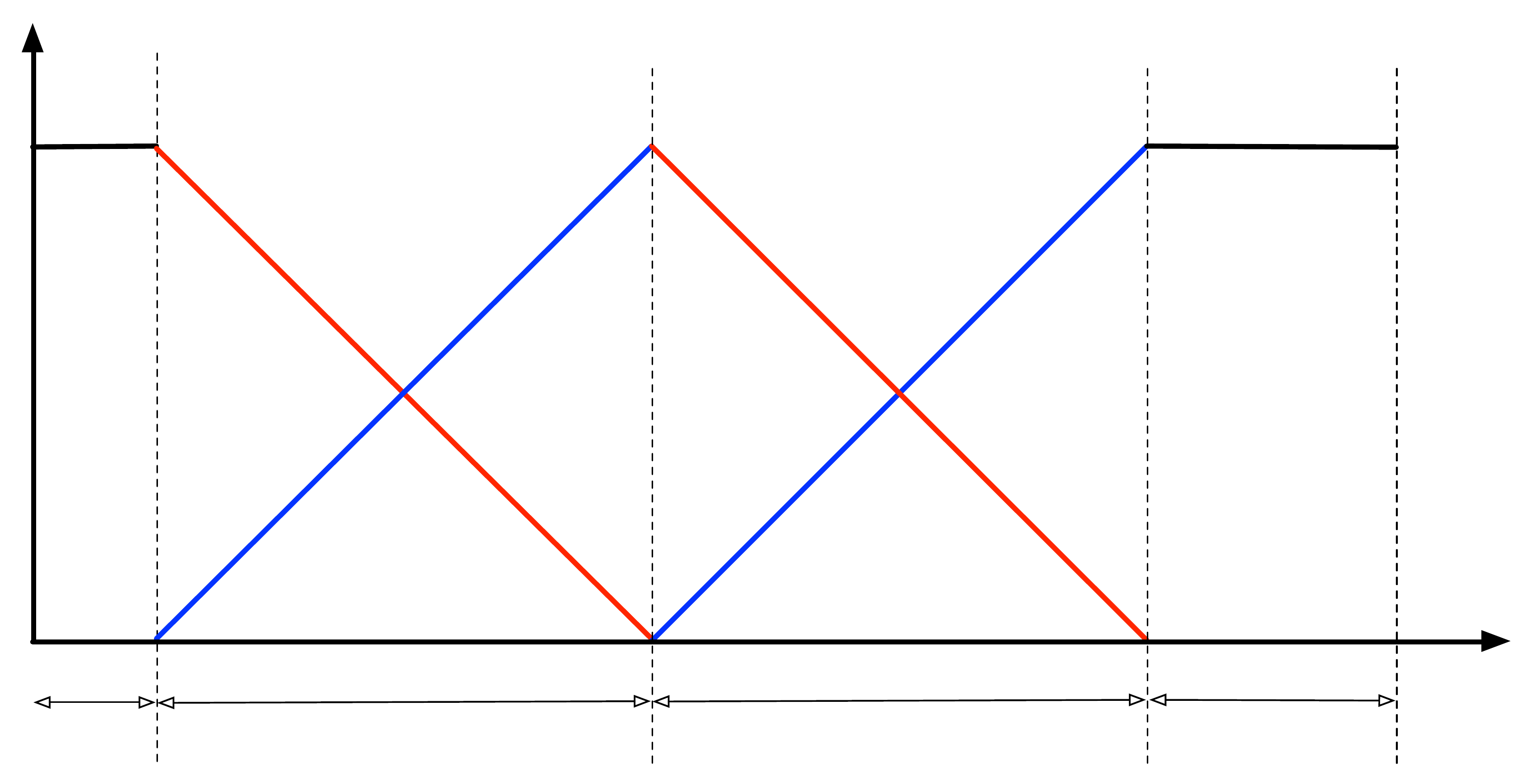}
		 \put(-275,25){\makebox(0,0)[l]{$0$}} 
		  \put(4,25){\makebox(0,0)[l]{$x$}} 
		  \put(-275,112){\makebox(0,0)[l]{$1$}} 
		  \put(-257,-1){\makebox(0,0)[l]{$\mathcal{D}_1$}} 
		   \put(-258,122){\makebox(0,0)[l]{$\phi_1^1(x)$}} 	
		  \put(-205,-1){\makebox(0,0)[l]{$\mathcal{D}_2$}} 
		  \put(-235,112){\makebox(0,0)[l]{$\phi_1^2(x)$}} 
		  \put(-185,112){\makebox(0,0)[l]{$\phi_2^2(x)$}}		
		  \put(-118,-1){\makebox(0,0)[l]{$\mathcal{D}_3$}} 
		  \put(-148,112){\makebox(0,0)[l]{$\phi_1^3(x)$}} 
		  \put(-98,112){\makebox(0,0)[l]{$\phi_2^3(x)$}}	
		  \put(-51,-1){\makebox(0,0)[l]{$\mathcal{D}_4$}} 	
		   \put(-57,122){\makebox(0,0)[l]{$\phi_1^4(x)$}} 			  			  	  	  
\caption{A stencil with nodal points $x_1 = 0, x_2 = 0.25, x_3 = 1.25, x_4 = 2.25, x_5 = 2.75$, and meshes $\mathcal{D}_1:=[0, 0.25], \mathcal{D}_2 := [0.25, 1.25], \mathcal{D}_3 := [1.25, 2.25], \mathcal{D}_4 := [2.25, 2.75]$, where $\mathcal{D}_1$ and $\mathcal{D}_4$ are referred to as \emph{boundary meshes} and $\mathcal{D}_2$ and $\mathcal{D}_3$ as \emph{interior meshes}. There is a constant basis function over each boundary mesh, and there are two linear basis functions over each interior mesh. \label{fig:example}}
\end{figure}

Then, we can verify that the matrices $M$ and $G$ are given by
\begin{align*} 
M & = \left[ \begin{array}{c|cc|cc|c} 
		1/4 &0 &0 &0 &0 &0 \\
	\hline
		0 & 1/3 & 1/6 & 0 &0 &0 \\
		0 & 1/6 & 1/3 & 0 & 0& 0\\
    \hline
		0& 0 & 0 & 1/3 & 1/6 & 0 \\
		0& 0 & 0 & 1/6 & 1/3 & 0 \\
    \hline
		0 & 0 & 0& 0& 0 & 1/2 
		\end{array} \right], \; G = \left[ \begin{array}{c|rr|rr|c} 
		0 &0 &0 &0 &0 &0 \\
\hline
			0 & -1/2 & 1/2 & 0 &0 &0 \\
			0 & -1/2 & 1/2 & 0 & 0& 0\\
\hline
			0& 0 & 0 & -1/2 & 1/2 & 0 \\
			0& 0 & 0 & -1/2 & 1/2 & 0 \\
\hline
			0 & 0 & 0& 0& 0 & 0 \end{array} \right].		
\end{align*} 

Assume that $c = 1$. Then, the non-zero upper-diagonal blocks are given by 
%the matrix $F$ is given by
%% 
%\begin{align*} 
%	F = \left[\begin{array}{cccc}
%	 F^{11} & F^{12} &         0  & 0 \\ 
%      0 & F^{22} & F^{23}  & 0  \\ 
%     	        0 & 0 &  F^{33}  &  F^{34}  \\ 
%	       0  &         0 &  0  &  F^{44} 
%	   	 \end{array} \right],
%\end{align*} 
%% 
%where the upper-diagonal blocks are 
%% 
\begin{align*} 
% k = 2
		F^{12} & = \eta_{1,2} \phi_1^{1}\left(x_{2}^{-}\right) 
		\left[\begin{array}{cc} 
		\phi_1^{2}\left(x_2\right) & \phi_2^{2}\left(x_2\right) 
		\end{array} \right] 
			 = \eta_{1,2} \left[\begin{array}{cc} 1 & 0 \end{array}\right],   \\
		 % 
% k = 3 			 
		 F^{23} & = \eta_{2,3}  \left[\begin{array}{c}
			\phi_1^{2}\left(x_{3}^{-}\right) \\
			\vspace*{-0.3cm} \\
			\phi_2^{2}\left(x_{3}^- \right) 
		\end{array}\right] 
		\left[\begin{array}{ccc} 
		\phi_1^{3}\left(x_3\right) & 
		\phi_2^{3}\left(x_3\right) 
		\end{array} \right]
%		= \eta_{2,3} 
%		\left[\begin{array}{c}
%			0 \\
%			1
%		\end{array}\right]
%		\left[\begin{array}{cc} 
%		1 & 0
%		\end{array} \right] 
		= \eta_{2,3} \left[\begin{array}{cc} 0 & 0 \\ 
				\vspace*{-0.3cm} \\	
		1 & 0  \end{array}\right] , \\
		% 
% k = 4 			 
		F^{34} & = \eta_{3,4} 	
		\left[\begin{array}{c}
			\phi_1^{3}\left(x_{4}^{-}\right) \\
			\vspace*{-0.3cm} \\
			\phi_{2}^{3}\left(x_{4}^{-} \right) 
		\end{array}\right] \phi_1^{4}\left(x_4\right) = \eta_{3,4} \left[\begin{array}{c} 0 \\ 1\end{array}\right], 
		\\
%		% 	
	\intertext{and the non-zero diagonal blocks are} 
%% k = 1 
		F^{11} & = -\phi_1^{1}\left(x_2^-\right)\phi_2^{1}\left(x_{2}\right) = -1, \\
		% 		
% k = 2 
		F^{22} & =   - \left[\begin{array}{c} 
		\phi_1^{2}\left(x_3^-\right) \\
		\vspace*{-0.3cm} \\ 
		\phi_2^{2}\left(x_3^-\right) 
		\end{array} \right] \left[\begin{array}{cc}
			\phi_1^{2}\left(x_3\right) & \phi_2^{2}\left(x_3\right) 
		\end{array}\right]  
%		=   - \left[\begin{array}{c} 
%		0 \\
%		1
%		\end{array} \right] \left[\begin{array}{cc}
%			0 & 1
%		\end{array}\right]  
		=   \left[\begin{array}{rr} 0 & 0 \\ 0 & -1 \end{array} \right], \\
		% 
	% 	\displaybreak
% k  = 3 
		F^{33} & = - \left[\begin{array}{c} 
		\phi_1^{3}\left(x_4^-\right) \\
		\vspace*{-0.4cm} \\ 
		\phi_2^{3}\left(x_4^-\right) 
		\end{array} \right] \left[\begin{array}{cc}
			\phi_1^{3}\left(x_4\right) &  \phi_2^{3}\left(x_4\right) 
		\end{array}\right]  
%		=  - \left[\begin{array}{c} 
%		0  \\
%		\vspace*{-0.4cm} \\ 
%		1
%		\end{array} \right] \left[\begin{array}{cc}
%			0 & 1
%		\end{array}\right]  
		=   \left[\begin{array}{rr} 0 & 0 \\ 0 & -1 \end{array}\right],   \\
		% 
% k  = 4 
		F^{44} & = - \phi_1^{4}\left(x_5^-\right) 
			\phi_1^{4}\left(x_5\right) =  - 1, 
			\end{align*} 
			with all other sub-blocks of $F$ being identically zero. Thus, 
\begin{align*} 
	F = \left[\begin{array}{r|cr|cr|r}
	 -1   & \eta_{1,2} &  0  & 0  & 0 & 0 \\ 
	 \hline
         0 &   0 &  0  & 0  & 0 & 0 \\
      0 &  0 & - 1 & \eta_{2,3} & 0  & 0      \\ 
 	 \hline     
     	        0 & 0 &  0  &   0 & 0  & 0 \\ 
	       0  &         0 & 0  &  0 & -1 & \eta_{3,4}  \\
	       	 \hline 
		       0  &      &   0 &  0  &  	&   - 1 
	   	 \end{array} \right],
\end{align*} 
where 
	\begin{align*}
		\eta_{1,2} & :=  \frac{\displaystyle\int_{\mathcal{D}_{2}} \phi_1^{2}(x) \ud x}{\displaystyle\int_{\mathcal{D}_1} \phi_1^{1}(x) \ud x} = 2, \;\; \eta_{2,3}  := \frac{\displaystyle\int_{\mathcal{D}_{3}} \phi_1^{3}(x) \ud x}{\displaystyle\int_{\mathcal{D}_2} \phi_1^{2}(x) \ud x} = 1, \;\; \eta_{3,4} := \frac{\displaystyle\int_{\mathcal{D}_{4}} \phi_1^{4}(x) \ud x}{\displaystyle\int_{\mathcal{D}_3} \phi_1^{3}(x) \ud x} = 1.  
	\end{align*} 
	Consequently, the DG generator $\mathcal{Q} = c (G + F)M^{-1}$ is given by  
	\begin{align*}
		\mathcal{Q}  =  \left[\begin{array}{r|rr|rr|r}
		 -4  &   8 &    -4 &     0 &    0 &    0 \\
		 \hline
		 0   &  -3   &  3 &    0  &   0 &    0 \\ 
            0  &   -1  &  -1  &  4   &   -2   &  0  \\
          \hline            
            0  &  0   &     0  &   -3 &  3 &    0 \\
           0  &   0   &  0 &  -1  &   -1 & 2 \\
          \hline 
            0  &   0 &    0  &   0 &      0 &    0 \end{array}\right],
	\end{align*} 
	where note that all the row sums are zero. 
	
	Now, to illustrate the approximation $\mathcal{B}$ of the operator matrix $\mathbb{B}$, suppose the phase process $\varphi_t \in \mathcal{S} = \{1,2\}$ has a generator matrix $T$. Let the DG approximation for this fluid-fluid process $\{X_t, Y_t, \varphi_t\}$ have the same stencil as described in Figure~\ref{fig:example}. 
	
	When $\varphi_t = 1$, we assume $Y_t$ has a positive fluid rate, $r_1(X_t) > 0$, when $X_t \in \mathcal{D}_1 \cup \mathcal{D}_2$, and a negative rate, $r_1(X_t) < 0$, when $X_t \in \mathcal{D}_3 \cup \mathcal{D}_4$. Thus, $\mathcal{F}^+_1 = \mathcal{D}_1 \cup \mathcal{D}_2$ and $\mathcal{F}^-_1 = \mathcal{D}_3 \cup \mathcal{D}_4.$ When $\varphi_t = 2$, we assume the opposite: $\mathcal{F}^{-}_2 = \mathcal{D}_1 \cup \mathcal{D}_2$ and $\mathcal{F}^+_2 = \mathcal{D}_3 \cup \mathcal{D}_4$. Thus, $
\gamma_{1}^{+} = \gamma_{2}^{-} =  \{1, 2\}$ and $\gamma_1^{-} = \gamma_2^{+} = \{3,4\}.$ 
	\begin{enumerate}
		\item[\textbf{Case 1.}] When $i \neq j$: Suppose $i = 1, j = 2, \ell = +$, and $m = -$. Then, 		% 
		\begin{align*}
			\mathcal{B}_{12}^{+-} = \left[\begin{array}{c|cc|cc|c} 
				T_{12} & 0 & 0 & 0 & 0 &  0 \\
				\hline
				0 & T_{12} & 0 & 0 & 0 &  0 \\			
				0 & 0 & T_{12} & 0 & 0 &  0 \\	
				\hline		
				0 & 0 & 0 & 0 & 0 & 0 \\					
				0 & 0 & 0 & 0 & 0 & 0 \\		
				\hline		
				0 & 0 & 0 & 0 & 0 & 0 \\
				\end{array}\right]. 
		\end{align*} 
%		\bleu{Only the first three bases have support in the set $\mathcal{F}^+_1 \cap \mathcal{F}^-_j$.} 
		% 
		\item[\textbf{Case 2.}] When $i = j$ and $\ell \neq m$: suppose $i = 1, \ell = +$, and $m = -$. Then, 
		\begin{align*}
			\mathcal{B}_{11}^{+-} = \left[\begin{array}{r|rr|rr|r} 
			0 & 0 & 0 & 0 & 0 & 0 \\
			\hline
			0 & 0 & 0 & 0 & 0 & 0 \\
			0 & 0 & 0 & 4 & -2 & 0 \\
               \hline			
			0 & 0 & 0 & 0 & 0 & 0 \\
			0 & 0 & 0 & 0 & 0 & 0 \\	
               \hline															
			0 & 0 & 0 & 0 & 0 & 0 \\
			\end{array} 
			\right]. 
		\end{align*} 
		If we are to pre-multiply $\mathcal{B}_{11}^{+-}$ with a row vector, we can see the matrix picks up the value of the right-most basis function with support in $\mathcal{F}_{1}^{+}$, and then distributes this to the first two basis functions with support in $\mathcal{F}_1^{-}$. {This concurs with our physical interpretation (see Section~\ref{subsec:B_operators}) that there is a drift from $\mathcal{F}_{1}^{+}$ to $\mathcal{F}_{1}^{-}$.}  
%	\intrusion{Should this be \emph{the first two basis functions in the negative phase}?} 
	\item[\textbf{Case 3.}] When $i = j$ and $\ell = m$: Suppose $i = 1$ and $\ell = -$. Then,   
		\begin{align*}
			\mathcal{B}_{11}^{--} = \left[\begin{array}{r|rr|rr|r} 
			0 & 0 & 0 & 0 & 0 & 0 \\
			\hline
			0 & 0 & 0 & 0 & 0 & 0 \\
			0 & 0 & 0 & 0 & 0 & 0 \\
			\hline
			0 & 0 & 0 & T_{11} & 0 & 0 \\
			0 & 0 & 0 & 0 & T_{11} & 0 \\		
			\hline
			0 & 0 & 0 & 0 & 0 & T_{11} \\
			\end{array}\right] + \left[\begin{array}{c|rr|rr|r} 
			0 & \;\;0 & \;\;0 & 0 & 0 & \;\;0 \\
			\hline
			0 & 0 & 0 & 0 & 0 & 0 \\
			0 & 0 & 0 & 0 & 0 & 0 \\
			\hline
			0 & 0 & 0 & -3 & 3 & 0 \\
			0 & 0 & 0 & -1 & -1 & 2 \\	
			\hline
			0 & 0 & 0 & 0 & 0 & 0 \\			
			\end{array}\right]. 
		\end{align*} 
	\end{enumerate}

\section{Numerical Experiments} 
	\label{sec:numexp}
	
To illustrate the validity of our discontinuous Galerkin approximation, we perform numerical experiments on a stochastic fluid-fluid model, in a three-pronged approach. 

First, we run Monte Carlo simulations, in order to compare the simulated joint density of $\{X_t, \varphi_t\}$ evaluated at the time $Y_t$ first returns to the initial level $0$ against that which is obtained via the return-probability matrix $\uppsi$. This numerically verifies the accuracy of our proposed approximation for the operator matrix $\Psi$.  Second, using $\uppsi$ we evaluate the limiting joint density of $\{X_t, \varphi_t \}$, which we compare against the same density analytically derived in~\cite{lnp13}.  Third, we vary the parameters of the second fluid $Y_t$ to confirm that the approximating joint density for $\{X_t, \varphi_t\}$ does not change, while the marginal limiting distribution for $Y_t$ does, both of which are consistent with our intuitive understanding of the chosen example. In all three procedures, we find the approximations to be accurate. 

We also analyse different choices for the level of spatial discretisation and the degree of polynomial basis functions, with respect to the order of convergence in relevant error terms. 

\subsection{An on-off bandwith-sharing model} 
The example we choose for our experiments is as follows. Consider a stochastic fluid-fluid $\{X_t,Y_t,\varphi_t\}_{t\geq0},$ where $X_t$ and $Y_t$ represent the workloads in Buffer~1 and Buffer~2 at time $t \geq 0$, both driven by the phase $\varphi_t,$ which is a Markov chain on the state space $\mathcal{S} = \{11,10,01,00\}$. Here, the state $11$ indicates inputs to both buffers being \textsc{on}, the state $00$ indicates both being \textsc{off}, the state $10$ is when only the first input is \textsc{on}, and the state $01$ is when only the second is \textsc{on}. The input of Buffer~$k$ is switched from \textsc{on} to \textsc{off} with rate $\alpha_k$, and from \textsc{off} to \textsc{on} with rate $\beta_k$, for $k = 1, 2$. Thus, the infinitesimal generator $T$ for $\varphi_t$ is given by 
\begin{align*} 
	T = \left[ \begin{array}{cccc} -(\alpha_1 + \alpha_2) & \alpha_2 & \alpha_1 & 0 \\
						\beta_2 & -(\alpha_1 + \beta_2) & 0 & \alpha_1 \\
						\beta_1 & 0 & -(\alpha_2 + \beta_1) & 0 \\
						0 & \beta_1 &\beta_2 &-(\beta_1 + \beta_2)
 \end{array}\right].
 \end{align*} 
 
We denote by $\lambda_k$ the input rate of Buffer~$k$ during an \textsc{on} period, and by $\zeta_k$ the output rate, and assume that $\lambda_k > \zeta_k$, for $k = 1, 2$. This example imitates the model considered in~\cite{lnp13}, where the two buffers share a fixed total output capacity $\kappa > 0$. 

In particular, the output rate $\zeta_1$ is constantly $\theta_1$, except for when the buffer is empty, that is, when $X_t = 0$. On the other hand, $\zeta_2$ varies depending on the values of the buffers. More specifically, as Buffer~1 is considered high-priority, this buffer is allocated the entire output capacity $\kappa$ whenever exceeding a pre-determined threshold $x^*$, leaving $\zeta_1 = \kappa$ and $\zeta_2 = 0$. However, when Buffer~1 is empty, Buffer~2 is given the whole output capacity $\kappa$, meaning $\zeta_2 = \kappa$. For $X_t \in (0, x^*)$, $\zeta_1 = \theta_1$ and $\zeta_2 = \theta_2$, for $\theta_1, \theta_2 \geq 0$ such that $\theta_1 + \theta_2 = \kappa$. Clearly, if Buffer~$k$ is empty, its output rate $\zeta_k$ would be zero. Table~\ref{table:outputs} summarizes the output rates for all different scenarios. 

 \begin{table}[h!]
 \centering
 \begin{tabular}{|c|c|c|c|}
 \hline
Buffer $1$               & Buffer $2$ and Phase & $\zeta_1$ & $\zeta_2$ \\
\hline 
$X_1 > x^*$            & $\{Y_t \in [0, \infty)$, $\varphi_t \in \mathcal{S}\}$  & $\kappa$                          & $0$                               \\ 
\hline
$ 0 < X_1 \leq x^*$ & $\{Y_t > 0, \varphi_t \in \mathcal{S}\}$ \mbox{ or } $\{Y_t = 0, \varphi_t \in \{11, 01\}\}$ & $\theta_1$ & $\theta_2$\\ 
 & $\{Y_t = 0, \varphi_t \in \{10, 00\}\}$ & $\theta_1$ & $0$\\ 
\hline 
$X_t  = 0$ & $\{Y_t > 0, \varphi_t \in \mathcal{S}\}$ \mbox{ or } $\{Y_t = 0, \varphi_t \in \{11, 01\}\}$ & 0 & $\kappa$ \\ 
 & $\{Y_t = 0, \varphi_t \in \{10, 00\}\}$ & $0$ & $0$\\ 
\hline  
 \end{tabular}
 \caption{Output rates $\zeta_1$ and $\zeta_2$ for the buffers in different scenarios, as specified in the model in \cite{lnp13}. Note that while Buffer~1 is independent of Buffer~2, its output rate $\zeta_1$ depends on $X_t$.  \label{table:outputs}}
 \end{table}

Given its dynamics as currently defined based on \cite{lnp13}, Buffer~1 is an example of what is known in the literature as a \emph{level-dependent} fluid, because its net rates depend on the value of $X_t$ relative to the threshold $x^*$. As the existing theoretical analysis for stochastic fluid-fluid processes (developed in \cite{bo2014} and briefly summarized in Section~\ref{sec:prelim}) does not allow {for level dependency in the first buffer}, here we modify the bandwith-sharing model in~\cite{lnp13} slightly. We let the output rate $\zeta_1$ of Buffer~1 remain $\theta_1$ for $X_t > 0$, effectively eliminating the threshold effect on the first buffer but keeping the effect on the second. Table~\ref{table:outputs-2} represents the modified rates. 

 \begin{table}[h!]
 \centering
 \begin{tabular}{|c|c|c|c|}
 \hline
Buffer $1$               & Buffer $2$ and Phase & $\zeta_1$ & $\zeta_2$ \\
\hline 
$X_1 > x^*$            & $\{Y_t \in [0, \infty)$, $\varphi_t \in \mathcal{S}\}$  & $\theta_1$                          & $0$                               \\ 
\hline
$ 0 < X_1 \leq x^*$ & $\{Y_t > 0, \varphi_t \in \mathcal{S}\}$ \mbox{ or } $\{Y_t = 0, \varphi_t \in \{11, 01\}\}$ & $\theta_1$ & $\theta_2$\\ 
 & $\{Y_t = 0, \varphi_t \in \{10, 00\}\}$ & $\theta_1$ & $0$\\ 
\hline 
$X_t  = 0$ & $\{Y_t > 0, \varphi_t \in \mathcal{S}\}$ \mbox{ or } $\{Y_t = 0, \varphi_t \in \{11, 01\}\}$ & 0 & $\kappa$ \\ 
 & $\{Y_t = 0, \varphi_t \in \{10, 00\}\}$ & $0$ & $0$\\ 
\hline  
 \end{tabular}
 \caption{Output rates $\zeta_1$ and $\zeta_2$ for the buffers in different scenarios, as modified slightly from the model in \cite{lnp13}. Note that the output rate $\zeta_1$ no longer depends on the value of $X_t$, except for a boundary condition at $0$. \label{table:outputs-2}}
 \end{table}

Consequently, the net rates of change for $X_t$, $c_i$, are given by 
\begin{align*} 
& (c_{11},c_{10},c_{01},c_{00})   = \left\{ \begin{array}{rrrrll} 
  (\lambda_1 - \theta_1, & \lambda_1 - \theta_1, &  0, & 0) & \text{if } X_t = 0, \\
  \vspace*{-0.3cm} \\
  (\lambda_1-\theta_1,  & \lambda_1 -\theta_1, &  -\theta_1, & -\theta_1) & \text{if } X_t > 0,
 \end{array}  \right.
\end{align*} 
and the net rates of change for $Y_t$, $r_i$, are as follows  
\begin{align*} 
& (r_{11},r_{10},r_{01},r_{00})  \\
& = \left\{ \begin{array}{lrrrll}  (\lambda_2 -\kappa, & \;\;\;\;0, & \lambda_2 - \kappa, &  \;\;\;\;0) & \text{if } X_t = 0,  Y_t = 0,\\
  \vspace*{-0.3cm} \\
  											(\lambda_2 - \kappa, & \;\;\;-\kappa,  & \lambda_2 - \kappa, & -\kappa) &\text{if } X_t = 0, Y_t > 0, \\
											  \vspace*{-0.3cm} \\
	(\lambda_2 -\theta_2,  &  0,  & \lambda_2 - \theta_2,  & \;\;\;\; 0) & \text{if } X_t \in (0,x^*), Y_t = 0,\\
											  \vspace*{-0.3cm} \\
	(\lambda_2 - \theta_2, & -\theta_2, & \lambda_2 - \theta_2, & -\theta_2) & \text{if } X_t \in (0,x^*), Y_t > 0,\\
											  \vspace*{-0.3cm} \\
	(\;\;\;\;\;\;\; \lambda_2, & \;\;\;\;0, & \lambda_2, &  \;\;\;\;0) & \text{if } X_t \geq x^*,  Y_t \geq 0. \end{array}
											\right.
\end{align*} 

For our numerical experiments, we use the parameter choices given in~\cite{lnp13}: 
	\begin{align} 
		\label{eqn:parameters}
	 \alpha_1 & =11, \quad  \beta_1 = 1, \quad \lambda_1 = 12.48, \quad  \theta_1 = 1.6, \quad  \kappa = 2.6, \\
	 	\label{eqn:parameters-2}
	\alpha_2 & = 22, \quad \beta_2  = 1, \quad  \lambda_2 = 16.25, \quad \theta_2 = 1.0, \quad x^* = 1.6.
	\end{align} 
%
% \begin{table}[h!]
% \centering
% \begin{tabular}{|l|l|l|l|l|}
% \hline
% $\alpha_1 =11$ &$ \beta_1 = 1$ &$ \lambda_1 = 12.48$ &$ \theta_1 = 1.6$&$ \kappa = 2.6$ \\
%\hline$  \alpha_2 = 22$ & $\beta_2  = 1$ &$ \lambda_2 = 16.25$ & $\theta_2 = 1.0$ &$ x^* = 1.6$\\
%\hline
% \end{tabular}
% \caption{Example coefficients \label{table:coeffs}}
% \end{table}
 
As mentioned previously, while the true problem has an unbounded domain $[0,\infty)$, the discontinuous Garlekin method requires the domain of approximation to be a finite interval. Hence, for all approximations we consider a finite interval but large enough so that the boundary-induced dynamics do not significantly affect {the} results.   
 
To specify the stencil for our numerical approximation, we define a vector $\bs{\omega}_{\rouge{K}, h, \Delta_h}$ of~$K$ nodal points as 
	\begin{align}
		\label{eqn:stencil}
%		\bs{\omega}_{h, \Delta_h} := \left(0, \Delta_h, h + \Delta_h, 2h + \Delta_h, \ldots, (K - 3)h + \Delta_h, (K - 3)h + 2\Delta_h\right),  
		\bs{\omega}_{\rouge{K}, h, \Delta_h} := \left(0, \Delta_h, \rouge{h}, 2h, \ldots, (K - 4)h, (K - 3)h - \Delta_h, (K - 3)h\right),  
	\end{align} 
for $h > 0$ and for $\Delta_h > 0$, both sufficiently small. \rouge{In this stencil, there are $K - 1$ meshes, of which $K - 5$ are interior meshes of length $h$, the left and right boundary meshes are of length $\Delta_h$, and the second-to-left and second-to-right ones are of length $h - \Delta_h$.} The boundary meshes always have piecewise-constant approximations, because this is sufficient to approximate the point masses accumulated at either boundary. 
%\intrusion{NB $\#1$: All figures to be reproduced with the new mesh?}  

\subsection{The return-probability matrix $\uppsi$ via Monte Carlo simulations}  
	\label{subsec:MCs} 

We choose our initial distribution $v_i(x,y,0) := \mathds{P}\left[X_0 \leq x, Y_0 \leq y, \varphi_0 = i\right]$  
to be a point mass of $1$ for $(X_0 = 5, Y_0 = 0, \varphi_0 = 01)$, and zero everywhere else. The choice of an initial point being a point mass instead of a non-degenerate distribution is purely for the convenience of numerical simulation, because it eliminates the need of simulating multiple initial starting points. Using this initial condition and the parameters specified in~(\ref{eqn:parameters}, \ref{eqn:parameters-2}), we simulate $10^5$ trajectories, terminating each path either when Buffer $Y_t$ returns to zero or at time ${V} = 10000$, Overall, 2.3\% of the trajectories did not reach zero in buffer $Y$ by time $V$, and so are rejected. 

Note that as we assume positive recurrence for our system, the probability of Buffer~2 returning to its initial level zero is $1$. However, the time this takes might be longer than the time window of our simulation, ${V} = 10000$, which means the trajectories that we have terminated at time ${V}$ must eventually return to zero with probability $1$.
%In other words, in terms of calculating, for example, the probability $\mathds{P}[X_{\tau} \leq x, Y_{\tau} = 0, \varphi_{\tau} = 00]$ from the simulated data, we take  
%% 
%	% 
%	\begin{align*} 
%		\frac{\# \mbox{trajectories in which $X_{\tau} \leq x$ and $\varphi_{\tau} = 00$} }{ \mbox{\# all trajectories (which is $10^6$)}}
%	\end{align*}}

Let $\tau := \inf \{t > 0: Y_t = 0\}$ be the first time $Y_t$ returns to zero. For each retained simulated trajectory, we record the values of $X_\tau$ and $\varphi_\tau$. For the states~$10$ (the first input being \textsc{on}, the second being \textsc{off}) and~$00$ (\textsc{off}-\textsc{off}), we present in Figure~\ref{fig:PSI_Cummu_Compare} the cumulative distributions of $X_\tau$, $\mathds{P}[X_{\tau} \leq x, Y_{\tau} = 0, \varphi_\tau = i]$, with $i = 00, 10$, as determined by the simulations as well as by a piecewise linear DG approximation constructed from the approximating matrix $\uppsi$ of operator $\Psi$. \bleu{In this piecewise linear DG approximation, we consider the approximation interval $[0, \mathcal{I}] = \rouge{[0,16]}$, and the parameters of the stencil~$\boldsymbol{\omega}_{\rouge{K}, h,\Delta_h}$, defined in~\eqref{eqn:stencil}, take the following values: $K = 43,  h = 0.4$ and $\Delta_h = 0.001$. The boundary meshes each have a constant basis function of value $1$; each $k$th interior mesh $\mathcal{D}_k := \rouge{[x_k, x_{k + 1}]}$ has two piecewise linear basis functions 
	\begin{align*}
		 \phi_0^{i}(x) := -x + \rouge{\frac{x_{k + 1}}{x_{k + 1} - x_k}}, \quad \phi_1^2(x) := x - \rouge{\frac{x_k}{x_{k + 1} - x_k}},
	\end{align*}
	for $x \in \mathcal{D}_k$ and $\rouge{k = 2, \ldots, K - 2}$.} As can be seen in Figure~\ref{fig:PSI_Cummu_Compare}, these two distributions are closely matched.

%\intrusion{Vikram: Can you please briefly describe the (exact) basis functions used in these approximations -- GN. \hp{ $h=0.1$ $\Delta_h = 0.001$ and piecewise linear basis in the middle. $\mathcal{I} = [0,14]$} }  
%\intrusion{Vikram: Another follow-up question: You mentioned that $[0,\mathcal{I}] = [0,14]$. It's clear that if we have $h = 0.1$ and $\Delta_h = 0.001$, the end point won't be $14$ if we want to have complete meshes. In particular, we could get either have 
%	\begin{itemize} 
%		\item $K = 141$, which allows us to have $139$ interior meshes of size $h = 0.1$, which means the final nodal point $K$ is at $2(0.001) + 139(0.1) = 13.902$, or 
%		\item $K = 142$, which allows us to have $140$ interior meshes of size $h = 0.1$, which means the final nodal point $K$ is at $2(0.001) + 140(0.1) = 14.002$.
%	\end{itemize}
%	% 
%	Could you please let me know which value of $K$ you used?} 

\begin{center}
	\includegraphics[width=\textwidth]{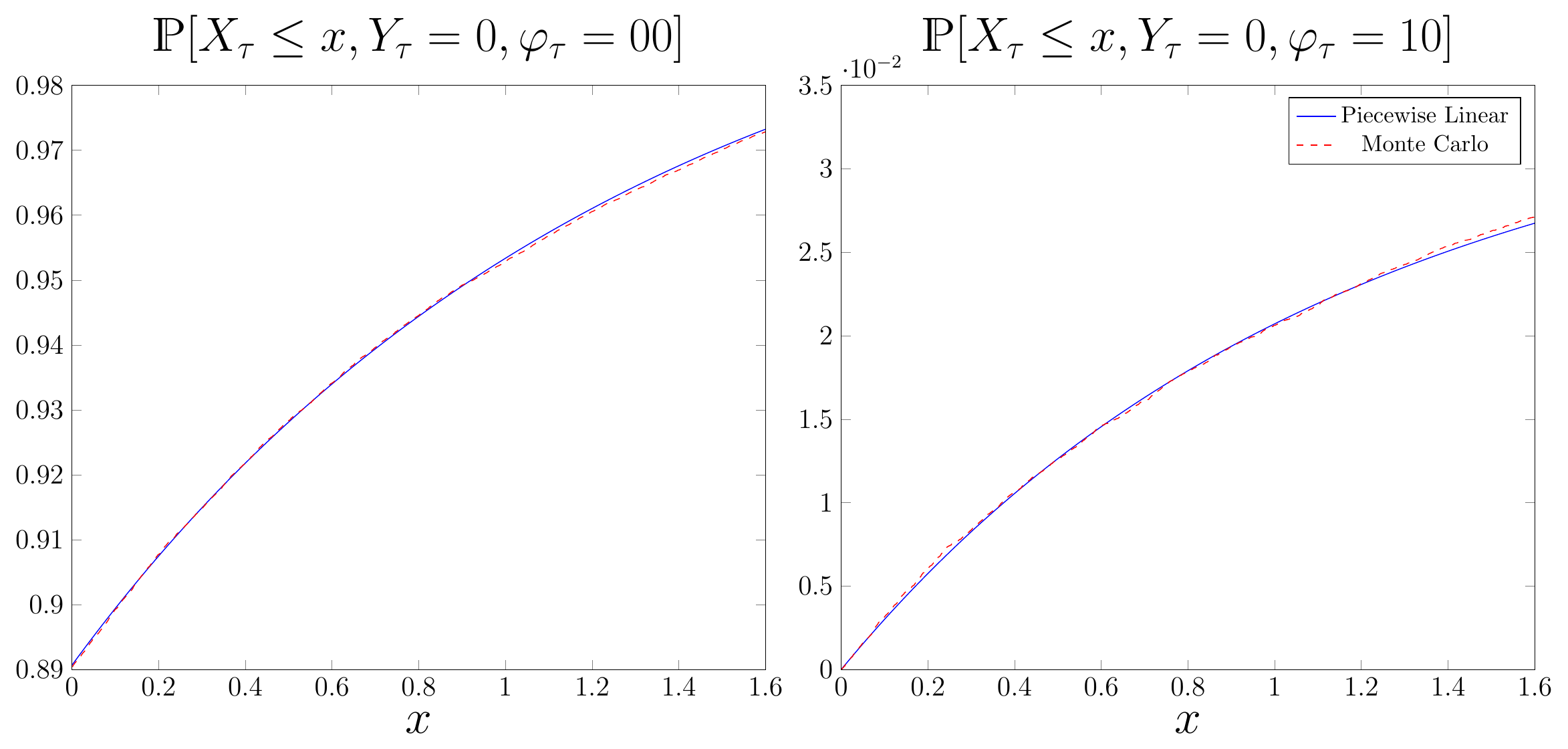} 
      \captionof{figure}{The {\textsc{off}-\textsc{off} and \textsc{on}-\textsc{off}} state cumulative distributions of the level $X_{\tau}$ with our chosen initial condition. In each plot, the \emph{solid blue} line is the piecewise linear DG approximation, and the \emph{dashed red} line is the empirical cumulative distribution at time $\tau$ of the stochastic process $\{X_t, Y_t, \varphi_t\}$, obtained from the retained simulations {and kernel density estimation}. The DG approximations appear to follow very closely the empirical cumulative distributions. \label{fig:PSI_Cummu_Compare}} 	
\end{center}

%We also plot in Figure~\ref{fig:PSI_Compare} the density functions of $X_{\tau}$, for state $10$ and for state $00$. Note that, here, the DG approximations of the density function of $X_{\tau}$ are smooth, which is not necessarily the case for DG approximations. 

%\intrusion{Might not be so surprising, we will discuss this further when we look at the error convergence later}
% 

%\begin{center}
%	\includegraphics[width=\textwidth]{Tikz/Psi_density.pdf} %scale=0.5
%      \captionof{figure}{The \textsc{on}-\textsc{off} and \textsc{off}-\textsc{off} densities of $X_\tau$ with our chosen initial condition. The \emph{solid blue} line is the piecewise linear DG approximation. The subplot inside the left plot shows the point mass at $X_{\tau} = 0$, which is approximately {0.76}. Note that the DG approximations of both densities are smooth. 
%      \label{fig:PSI_Compare}} 	
%\end{center} 

\subsection{The marginal density of $\{X_t, \varphi_t\}$}  
	\label{subsec:vmd}

Since Buffer~1, $X_t$, is independent of Buffer~2, $Y_t$, we can use results from the existing literature on stochastic fluid flows to obtain the marginal limiting density $\bs{\chi}(x) = (\chi_i(x))_{i \in \mathcal{S}}$ of $X_t$: 
	\begin{align*} 
		\chi_i(x) :=  \frac{\partial}{\partial x} \lim_{t \rightarrow \infty} \mathds{P}[X_t \leq x, \varphi_t = i].
	\end{align*} 
	
	On the other hand, we can use the operator $\Psi$ to compute, based on Theorem~\ref{theo:density}, the joint limiting density $\pi_i(\mathcal{A},y)$, where 
	\begin{align*}
		\pi_i(y)(\mathcal{A}) :=  \frac{\partial}{\partial y} \lim_{t \rightarrow \infty} \mathds{P}[X_t \in \mathcal{A}, Y_t \leq y, \varphi_t = i].
	\end{align*} 
	
	Thus, we can approximate the joint limiting density $\pi_i(y)(\mathcal{A})$, via a discontinuous Galerkin approximation $\uppsi$ of $\Psi$, and consequently form an approximation of the marginal limiting density $\bs{\chi}(x)$. In particular, we evaluate an approximation of $\pi_i(y)([0,x))$, and then integrate the approximating function over the domain of $y$, and finally differentiate with respect to $x$; note that  
	\begin{align*} 
		\frac{\partial}{\partial x} \int_{0}^{\infty} \pi_i(y)([0,x)) \ud y = \frac{\partial}{\partial x} \lim_{t \rightarrow \infty} \mathds{P}[X_t \in [0,x), Y_t \leq \infty, \varphi_t = i].
	\end{align*} 
%	\intrusion{Vikram: Just to be sure, is this what you did numerically? -- GN \hp{YES}} 
	
	Let two vectors $\overline{\bs{\chi}}$ and $\widehat{\bs{\chi}}$ denote respectively the piecewise constant and piecewise linear DG approximations of $\boldsymbol{\chi}$, obtained via the corresponding approximations $\uppsi$. We use $\bs{\omega}_{\rouge{43},0.4,0.001}$ as our stencil for the DG approximation and the nodal points at which we evaluate the analytical density function~$\boldsymbol{\chi}$. Define 
	\begin{align*} 
		\chi_{\textsc{on}}(x) & := \frac{\partial}{\partial x} \lim_{t \rightarrow \infty}  \mathds{P}\left[X_t \leq x, \varphi_t \in \{10, 11\}\right] = \chi_{10}(x) + \chi_{11}(x), \\
		\chi_{\textsc{off}}(x) & := \frac{\partial}{\partial x} \lim_{t \rightarrow \infty} \mathds{P}\left[X_t \leq x, \varphi_t \in \{01, 00\}\right] = \chi_{01}(x) + \chi_{00}(x).	
	\end{align*}
	
	We present in Figure~\ref{fig:on_off_limit} the analytical density $\bs{\chi}$ at given nodal points, a piecewise constant DG approximation, and a piecewise linear DG approximation.
	
\begin{center}
	\includegraphics[width=\textwidth]{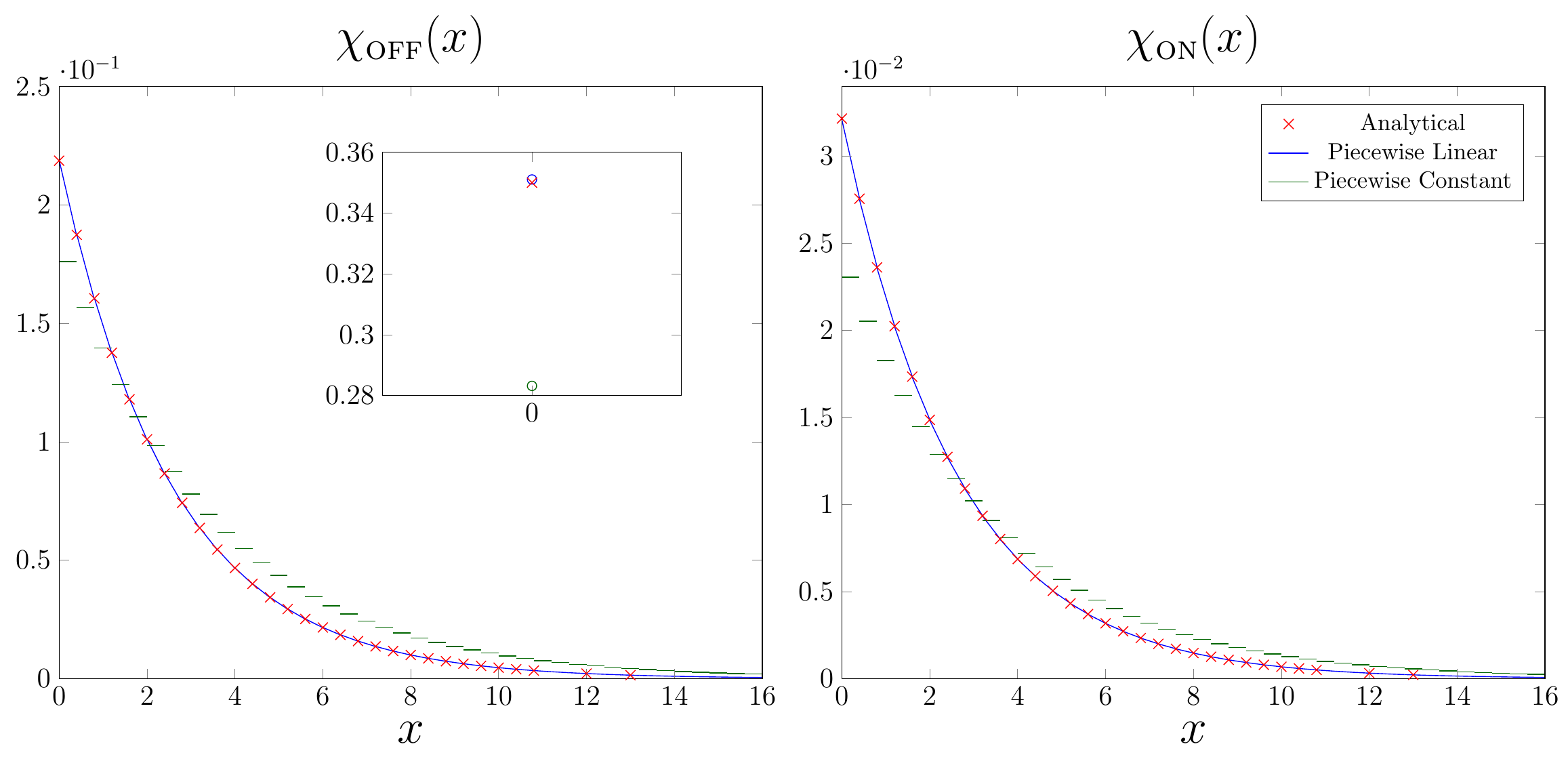}
       \captionof{figure}{Approximations of the stationary marginal densities $\chi_{\textsc{on}}$ and $\chi_{\textsc{off}}$ of $\{X_t\}$. In the left plot, the \emph{red crosses} are the values of the analytical solution $\chi_{\textsc{on}}$, evaluated at nodal points; the \emph{horizontal green} lines are the piecewise constant DG approximation $\overline{\chi}_{\textsc{on}}$; and the \emph{solid blue} line is a piecewise linear approximation $\widehat{\chi}_{\textsc{on}}$. In the right plot are the corresponding graphs for $\chi_{\textsc{off}}$. The subplot inside the left plot shows the value of the point mass of $X_t$ at $0$ and their DG approximations. The linear DG approximations fit through all analytical values to the visual eye. \label{fig:on_off_limit}}
       % ; the subplot in the right is a magnification of the plot for $x \in [0,1.5]$.} 
\end{center} 
%
%\intrusion{Vikram: I was wondering why the left plot has two fewer crosses than the right plot? Technically, they should have the same number of nodal points, right? I looked at the tikz files, and for some reason after $x \geq 10$ the code for the right one has different nodal points than the right one... We should fix this. -- GN \hp{ I just removed the points so you could see the line at the bottom, but didn't remove them evenly in the other one too}} 

%\intrusion{Vikram: I've commented out the subplot inside the right plot, which was to demonstrate that there's continuity in the function but not in its derivative. I wasn't sure if there is actually continuity in the function, because if you comment out the red crosses the blue line isn't continuous -- there are discontinuities at nodal points. How do we determine if this is an artifact of the plotting procedure, or the approximating function is indeed discontinuous? \hp{The approximation is discontinuous in construction. There is no plotting artefact. To preserve mass, we need to give up continuity, in the DG we try to give this up only at the nodal points. (key point on why we choose DG) Maybe best to bring up this point again here, in the case the reader forgot.}} 

%\intrusion{{Vikram, could you please insert legends to the top-right corner for the right subplot, as suggested by Nigel in his commented version?} \hp{DONE!!}}

The approximations reconstruct the general shape of the density reasonably well: for both piecewise constant and piecewise linear, we see most of the probability being concentrated in the point mass at $X_t = 0$, and then decaying as $x$ increases. We observe from the subplot inside the left plot of Figure~\ref{fig:on_off_limit} that the piecewise constant approximation $\overline{\chi}_{\textsc{off}}$ underestimates the {point} mass, and redistributes the difference over the rest of the state space of $X_t$. Hence, there is more mass in the tails of the densities. On the other hand, the piecewise linear approximation $\widehat{\bs{\chi}}$ appears to be very close to the analytical solution. 
%\rouge{Similar to the densities evaluated at the first-return time~$\tau$, the density is continuous, but we can see in the zoomed plot in the right figure that the derivatives are not continuous at nodal points. 
%However, they do connect on either side, making the density look like an element of $|\mathcal{S}| \times \mathcal{C}^{0}$.} 
%\intrusion{Vikram: The same comments about continuity, as mentioned above, apply here... Could you please clarify. -- GN \hp{ Same as above}} 

\subsection{Sensitivity analysis of the dynamics of $Y_t$}  
	\label{subsec:sensitivity}
	% \intrusion{Giang: compare with analytical bounds!} 
To further confirm that the discontinuous Galerkin approximation $\uppsi$ of the operator $\Psi$ accurately captures the dynamics of $Y_t$, we vary the rates at which the input to this buffer switches \textsc{on} and \textsc{off} (denoted by $\alpha_2$ and $\beta_2$, respectively). As we modify these rates, we should see a change in the distribution of probability between
	\begin{align*} 
		\chi^{0}(x) := \frac{\partial}{\partial x} \lim_{t \rightarrow \infty}  \mathds{P}\left[X_t \leq x, Y_t = 0\right], 
	\end{align*} 
 the limiting marginal density of $X_t$ when $Y_t = 0$, and 
	\begin{align*}
		 \chi^{+}(x) :=  \frac{\partial}{\partial x} \lim_{t \rightarrow \infty}  \mathds{P}\left[X_t \leq x, Y_t > 0\right], 		 
	\end{align*} 
that of $X_t$ when $Y_t > 0$. On the other hand, the sum of these two densities should be identical in all the different meaningful scenarios of $\alpha_2$ and $\beta_2$; that is, the sum ${\chi}^{0}(x) + {\chi}^{+}(x)$ {should remain fixed and be} equal to the sum ${\chi}_{\textsc{on}}(x) + {\chi}_{\textsc{off}}(x)$, {for all $x$.}  

To that end, we keep our stencil $\bs{\omega}_{43,0.4, 0.001}$ and basis functions fixed, and compute the marginal limiting density for different values of $\alpha_2$ and $\beta_2$. The results coincide with what we expect from the dynamics in $Y_t$. 

\begin{table}[!h]
\centering
\begin{tabular}{|l c | c | c | c | }
\hline
  & & $\alpha_2 = 11,\ \beta_2 = 1$ &$\alpha_2 = 16,\ \beta_2 = 1$  & $\alpha_2 = 22,\ \beta_2 = 1$ \\
\hline
$\displaystyle \int_{[0,\mathcal{I}]} \widehat{\chi}^0(x) \ud x$ & & $\approx$ 0.0 & 0.184  &0.312  \\
\hline
$\displaystyle \int_{[0,\mathcal{I}]} \widehat{\chi}^{+}(x) \ud x$& & $\approx$ 1.0 & 0.816 &0.688  \\ 
\hline
\end{tabular}
\caption{The functions $\widehat{\chi}^0(x)$ and $\widehat{\chi}^{+}$ are piecewise linear DG approximations of the limiting marginal densities $\chi^0(x)$ and $\chi^{+}(x)$ over the stencil $\bs{\omega}_{\rouge{43}, 0.4,0.001}.$  \label{table:varying_rates}}
\end{table}

From Table~\ref{table:varying_rates}, we observe that as $\alpha_2$ (the rate at which the input for $Y_t$ switches \textsc{off}) increases, so does the probability of $Y_t$ being empty. Furthermore, even though there are different amounts of probabilities in the two marginal densities, $\widehat{\chi}^0(x)$ and $\widehat{\chi}^{+}(x)$, their sum remains the same as the sum of the marginal limiting densities $\widehat{\chi}_{\textsc{on}}(x)$ and $\widehat{\chi}_{\textsc{off}}(x)$ for all calculated values of $x$ (data not shown here). These numerical results indicate that the dynamics of $Y_t$ are captured well by the DG approximations. 

%\begin{center}
%	\includegraphics[width=\textwidth]{Tikz/Psi_density.pdf} %scale=0.5
%      \captionof{figure}{The \textsc{on}-\textsc{off} and \textsc{off}-\textsc{off} densities of $X_\tau$ with our chosen initial condition. The \emph{solid blue} line is the piecewise linear DG approximation. The subplot inside the left plot shows the point mass at $X_{\tau} = 0$, which is approximately {0.76}.  
%      \label{fig:PSI_Compare}} 	
%\end{center} 

%\intrusion{Vikram: I've commented out the sentence ``Furthermore, we have computed the marginal limiting density via a multitude of non-linear steps, \eqref{eqn:RiccatiPsi}, and the approximation is still consistent with the analytical density.'' Wasn't sure whether it was worth including, because we haven't really made a big deal out of non-linear steps. Happy to be convinced otherwise. -- GN. \hp{I am with you on this one, I think we should leave this for another day}}

\subsection{Errors of approximation}  
	\label{subsec:errors}

In a discontinuous Galerkin approximation, we choose the smoothness of the basis functions {and} the level of spatial discretisation. Once a particular selection is made, we project our operators into a finite-dimensional linear operator space corresponding to these choices {(see Section~\ref{sec:numframe})}. It has been shown that operators such as $\mathbb{B}$ {(defined in Section~\ref{subsec:B_operators})} under a DG approximation have an error which converges at the order of $\mathcal{O}(h^{s}),$ where $h$ is the discretisation and $s$ is the degree of the basis~\cite{riviere2008discontinuous}.

%\intrusion{Vikram: Could you please give a reference for the above result. -- GN. \hp{I did, but bibtex doesn't like me} } 
 
However, this result cannot be easily translated across to the operator $\Psi.$ The DG approximation of the $\Psi$ operator is constructed by taking the DG approximation of the operators $\mathbb{B}$ and then solving the Riccati equation~\eqref{eqn:RiccatiPsi} using the approximate operators. Further, we then use this approximation of $\Psi$ to derive an approximation for the limiting density $\bs{\pi}$. With such a construction, it is not trivial to determine how the error propagates through the process of solving the Riccati equation, and then through further calculations to determine $\bs{\pi}$. Determining bounds for the approximation errors of $\Psi$ {and $\bs{\pi}$}, as functions of the discretisation and basis selection, is a topic for future research. 

As a preliminary step in this direction, we empirically investigate how the approximation error of the marginal limiting density of the fluid $X_t$~(see Section~\ref{subsec:vmd}) changes with respect to the choices of basis functions and the levels of discretisation. We begin by introducing our normed vector space in which we compare the different levels of discretisation and smoothness. Recall, from~\eqref{eqn:stencil}, that the left boundary mesh of our approximations is of length~$\Delta_h$, and all but two interior meshes are of length~$h$, with two compensating meshes of width $h - \Delta_h$. 
% 
%\intrusion{Vikram: Sorry for being slow here, but could you please clarify this part to me. As far as I can see, let's say our stencil is $\bs{\omega}_{h, \Delta_h} := \bs{\omega}_{0.4,0.001}$, where there are $K - 3$ interior meshes, each of size $h$, and $2$ boundary meshes, each of size $0.001$. The analysis below takes the interval $[0.001, \mathcal{I}]$, divides it into $W$ intervals of length~$d$, and then integrates a density over \emph{one} (new) mesh of size~$d$. 
%\hp{Yes, due to the addition of the boundary state, $h = (\mathcal{I} - 2\Delta_h)/W$.} 
%\\
%My questions are: 
%\begin{enumerate} 
%\item In the definition~\eqref{eqn:Pdg} below, is the mapping $\mathcal{P}_d$ supposed to be a vector? If so, I'll state that explicitly. \hp{Yes they are, you should do that} \\
%\item With this ``re-definition" of the meshes, are we effectively lumping together the error of the boundary created by the truncation of domain, with the error inside the interior $[\Delta_h, (K - 3)h + \Delta_h]$?
%\end{enumerate}} 
Let $[0,\mathcal{I}]$ be the interval on which we approximate our solution, then both the approximations and the analytical solution belong to the space $\mathcal{S} \times \mathcal{C}^{-1}([0,\mathcal{I}])$, where $\mathcal{C}^{-1}([0,\mathcal{I}])$ is the set of functions with countably many discontinuities. We choose the right boundary mesh to be a piecewise-constant function. Then, for any function $g: \mathcal{S} \times [0,\mathcal{I}] \mapsto \mathds{R}$, where $g(i, \cdot) \in  \mathcal{C}^{-1}([0, \mathcal{I}])$ for $i \in \mathcal{S}$, we define the \emph{star seminorm} as follows:
	  \begin{align}
	  	\label{eqn:Pdg}
	\Vert g \Vert_{\star} := \sum_{i\in \mathcal{S}}  \left\vert   \int_0^{\Delta_h} g(i,x) dx  \right\vert  + \int_{\Delta_h}^{ \mathcal{I} - \Delta_h }   \vert g(i,x) \vert dx +  \left\vert  \int_{ \mathcal{I} - \Delta_h}^{\mathcal{I}} g(i,x) dx  \right\vert.
	  \end{align}
%

% 
%\intrusion{NB$\#2$: What does this mean? $\approx I - h, I - \Delta_h, I/2$? \hp{Changed} }
% 
Essentially, the star seminorm is an extension of the $L^1$ norm which incorporates our interpretation: that the left and right boundary meshes are treated as point masses. That is, we only study the total mass in the intervals $[0,\Delta_h]$ and $[\mathcal{I}-\Delta_h,\mathcal{I}]$ and not the distribution over them.
% 
%\intrusion{NB$\#3$: No we don't, as we still use abs$(\cdot)$. No change of sign is correct! \hp{Changed} } 

We conduct two numerical experiments. The first is to understand the effects of choosing piecewise-linear basis functions over piecewise-constant, the second experiment is to understand the effects of treating a point mass as a density on a short interval. 

In the first experiment, we choose $\Delta_h = 10^{-6}$ and  $\mathcal{I}=16$. We then consider the error between the approximation and the reference solution in the star seminorm. \rouge{The notations $\widehat{\chi}_{\bs{\omega}}$ and $\overline{\chi}_{\bs{\omega}}$ denote, respectively, a piecewise-linear approximation with stencil $\bs{\omega}$ and a piecewise-constant approximation with stencil $\bs{\omega}$.}

In the left plot in Figure~\ref{fig:discretisation_error}, we observe that the piecewise-constant approximation has an error that scales approximately $\mathcal{O}({h^{0.88}}),$ with respective to the mesh size~$h$, and the piecewise linear approximation has an error that scales approximately $\mathcal{O}({h^{1.84}})$. We observe that the coarsest piecewise-linear basis approximation has an error similar to that of the finest piecewise-constant approximation, which is two orders of magnitude finer.

In the second experiment, we fix $h = 1.0,$  $\mathcal{I} = 16$, and the basis functions to be piecewise-linear, and then we scale $\Delta_h$ and observe the trend in the star seminorm. The approximation error will plateau past the length where the error caused by $h$ and $\mathcal{I}$ dominate over the error gains of reducing the width of the boundary mesh. Hence, we subtract from this approximation error the approximation error of a finer approximation, $ \widehat{\chi}_{19,1.0,0.005}$. The right plot in Figure~\ref{fig:discretisation_error} shows that the error scales approximately $\mathcal{O}({\Delta_h^{1.7}})$.  
\begin{center}
	\begin{figure}[!h]
	 \includegraphics[width=\textwidth]{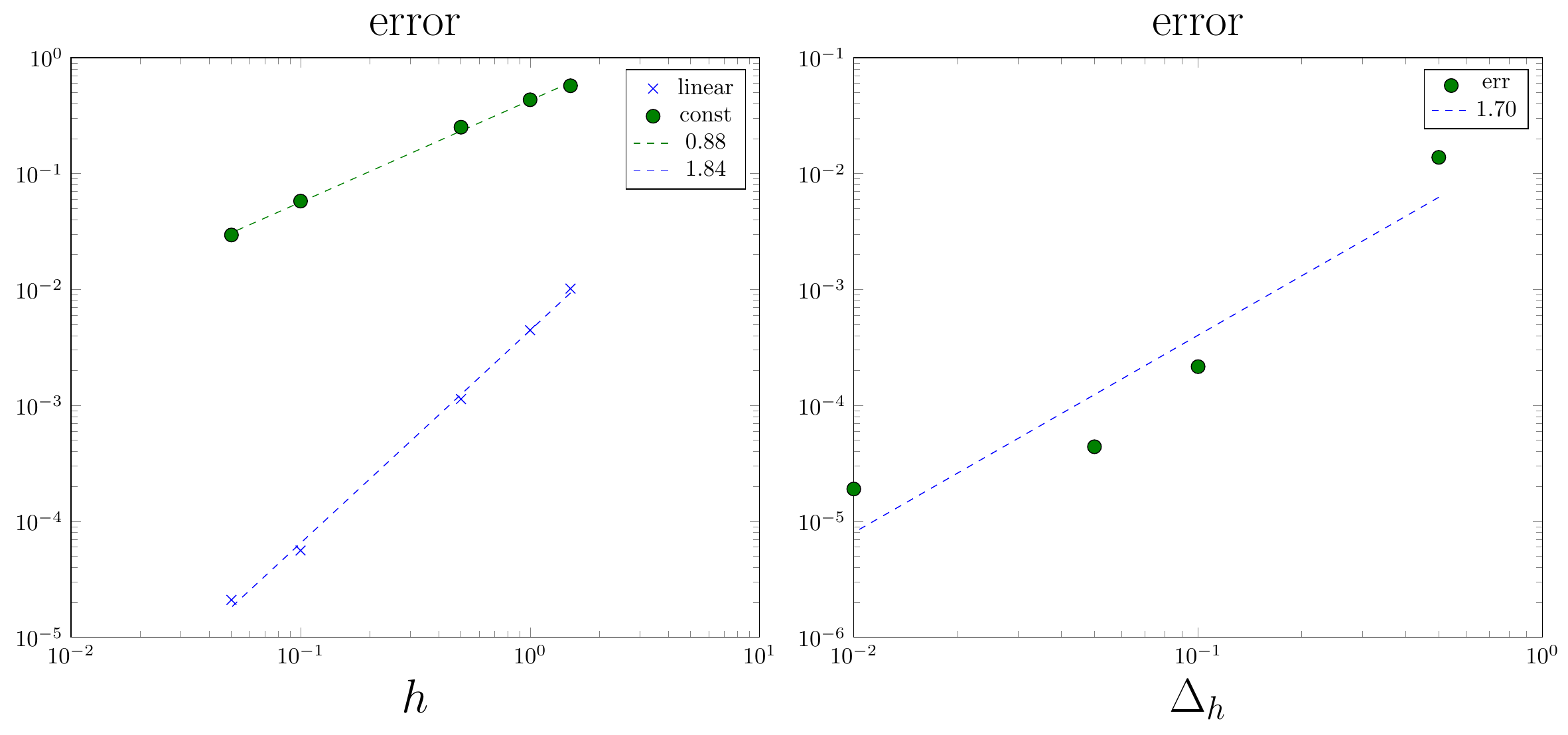}
	 % 
	% \caption{Approximation errors of different mesh sizes. On the left, the \emph{green crosses}, connected by the \emph{dashed green} line, represent the error $\Vert \chi -  \overline{\chi}_{h,10^{-6}} \Vert_\star$ of the piecewise constant approximations of different interior mesh sizes, $h$; the \emph{blue crosses} (connected by the \emph{solid blue} line) represent the error $\Vert \chi - \widehat{\chi}__{h,10^{-6}}  \Vert_{\star}$ of the piecewise linear approximations for different interior mesh sizes, $h$. 
	 \caption{Approximation errors of different mesh sizes. On the left, the \emph{green crosses}, approximated by the \emph{dashed green} line, represent the error $\Vert \chi -  \overline{\chi}_{h,10^{-6}} \Vert_\star$ of the piecewise constant approximations $\overline{\chi}_{h,10^{-6}}$ of different interior mesh sizes, $h$;  the \emph{blue crosses}, approximated by the \emph{solid blue} line, represent the error $\Vert \chi - \widehat{\chi}_{h,10^{-6}}  \Vert_{\star}$ of the piecewise linear approximations $\widehat{\chi}_{h,10^{-6}}$ for different interior mesh sizes, $h$.  The approximation error for a piecewise constant approximation is $\mathcal{O}\left(h^{0.88}\right)$, and the error for a piecewise linear approximation is $\mathcal{O}\left(h^{1.84}\right)$. On the right, the green crosses represent the error of the point mass for different boundary mesh sizes, $\Delta h$; note that  $\vert \Vert \chi - \widehat{\chi}_{K,1.0,\Delta h} \Vert_{\star} - \Vert \chi - \widehat{\chi}_{K,1.0,0.005} \Vert_{\star} \vert  = \mathcal{O}\left(\Delta_h^{1.7}\right)$.  \label{fig:discretisation_error}}
	\end{figure}
\end{center} 

The numerical experiments above indicate that an increase in the degree of the basis functions would result in an increase in the order of convergence of the error in the star seminorm. However, this experiment is not taking into consideration that by increasing the degree of basis functions we need to increase the number of basis (elements) in the mesh. For example, the piecewise-constant has one basis element in each mesh, while the piecewise-linear has two. With respect to storage, the overall number of elements in the entire stencil is increasing linearly with the order of the elements. 

In Table~\ref{table:times}, we give some computational statistics of the approximations used in Figure~\ref{fig:discretisation_error} to aid in the understanding of the trade-offs between memory, computation time, mesh size, and error. We observe for a mesh size $h = 0.5$ that the piecewise-constant approximation uses roughly half the number of elements and is twice as fast, compared to the piecewise linear approximation. 
However, the piecewise-linear approximation is two hundred times more accurate than the piecewise-constant approximation. 

In our particular examples, if we were given a restriction to use no more than a prescribed amount of elements, then choosing a larger mesh size with piecewise-linear elements would be a better strategy than choosing a smaller mesh size and using piecewise-constant elements. The generalisation of these observations and the exploration of higher order basis functions is the focus for future research.
\begin{table}[!h]
\centering
\begin{tabular}{| c || c ||  c | c | c | c | }
\hline
Basis & $h$ & Error & Comp. Times & \# of Elements & Overall Storage \\
\hline
\multirow{3}{4em}{\centering Piecewise Constant} & 1.5 & 0.58 & 0.21 sec & 88 & 0.5 MB \\
					& 0.5 & 0.25&0.31 sec & 248 & 4.1 MB \\
					& 0.05&0.03 & 23 sec & 2408 &380 MB \\ 
					\hline
\multirow{3}{4em}{\centering Piecewise Linear} & 1.5 & 0.01 & 0.21 sec & 168 & 2.1 MB \\
					& 0.5 & $1.1\times10^{-3}$ & 0.78 sec & 488 & 16 MB \\
					& 0.05 &  $2.1\times10^{-5}$ & 130 sec & 4816 & 1.5 GB  \\ 
\hline
\end{tabular}
\caption{ Computational times and storage comparisons between piecewise-linear and -constant approximations. 
Overall storage is the total storage of all the operators from \eqref{eqn:VB} to \eqref{eqn:mass}. 
The computations were performed on 2.5Ghz Intel Core i7 with 16GB of RAM running OSX 10.10.5. The code was implemented in python, using scientific python libraries. \label{table:times}}
\end{table}

%\intrusion{NB $\#4$: Should we include some timing data as well? To show the computational effort and how it grows with $h, \Delta_h, K$, constant/linear? \hp{Maybelline} }

%      \begin{minipage}{0.45\linewidth}
%              \input{photos/PM_error.tikz}
%      \end{minipage}
%%       \captionof{figure}{Boundary plots. \label{fig:discretisation_error}}
%\intrusion{(a) Vikram, is the $x$-axis for the graph on the right supposed to be $\Delta_ h$? If so, could you please change it? (b) Also, could you please have the top line being dotted, and the bottom being continuous, so that the legends don't have to be colour-dependent?}

\section{Conclusions} 
Finite Differences and Finite Volume methods have been used in the past to approximate operators that arise in stochastic processes. In principle, these methods approximate the operators by higher dimensional linear operators.
In these methods, intuitive notions of mass conservation and positivity are captured; however, regularity is lost, making highly regular probability distributions computationally intensive. 

We proposed the application of the discontinuous Galerkin method to approximate stochastic operators, with the intent that its ability to incorporate local regularity and maintain mass conservation, will lead to more accurate approximations and a reduction in computational effort. To demonstrate this, we applied the discontinuous Galerkin method to approximate all the operators needed to construct the joint stationary distribution of a stochastic fluid-fluid process. 

The numerical results showed that the approximation of the stationary distribution arising from DG approximations of the operators is accurate and effective. We also verified that the operators and their dynamics were captured accurately. Furthermore, in our example, we observed that adding more regularity in the basis functions led to a significant decrease in computational effort. The DG method also enabled us to obtain other performance measures of stochastic fluid-fluid processes that are also analytically presented by operators.

Future work includes determining error bounds for the approximations of the operator $\Psi$ as well as of the stationary distribution in general, and a more thorough investigation of the computational effort as higher-order basis functions are used. 

\section*{Acknowledgements}
The authors acknowledge the financial support of the
Australian Research Council (ARC) through the Discovery Grants DP110101663 and DP180103106. Bean, Nguyen, and O'Reilly also acknowledge the support of ACEMS (ARC Centre
of Excellence for Mathematical and Statistical Frontiers).

\bibliography{stochff}

\begin{thebibliography}{1}
\expandafter\ifx\csname url\endcsname\relax
  \def\url#1{\texttt{#1}}\fi
\expandafter\ifx\csname urlprefix\endcsname\relax\def\urlprefix{URL }\fi
\expandafter\ifx\csname href\endcsname\relax
  \def\href#1#2{#2} \def\path#1{#1}\fi

\bibitem{mz2012}
M.~Miyazawa, B.~Zwart, Wiener-{H}opf factorizations for a multidimensional
  {M}arkov additive process and their applications to reflected processes,
  Stochastic Systems 2 (2012) 67--114.

\bibitem{bo2014}
N.~G. Bean, M.~M. O'Reilly, The stochastic fluid-fluid model: A stochastic
  fluid model driven by an uncountable-state process, which is a stochastic
  fluid itself, Stochastic Processes and their Applications 124 (2014)
  1741--1772.

\bibitem{neuts81}
M.~Neuts, Introduction to Matrix Analytic Methods in Stochastic Modeling, The
  John Hopkins University Press, 1981.

\bibitem{lr99}
G.~Latouche, V.~Ramaswami, Introduction to matrix analytic methods in
  stochastic modeling, ASA-SIAM Series on Statistics and Applied Probability,
  SIAM, Philadelphia PA, 1999.

\bibitem{c99}
B.~Cockburn, Discontinous {G}arlekin methods for convection-dominated problem,
  in: {H}igher-{O}rder {M}ethods for {C}omputational {P}hysics, Vol.~9 of
  Lecture Notes in Computational Science and Engineering, Springer Verlag,
  1999.

\bibitem{lnp13}
G.~Latouche, G.~T. Nguyen, Z.~Palmowski, Two-dimensional fluid queues with
  temporary assistance, Vol.~27 of Springer Proceedings in Mathematics \&
  Statistics, Springer Science, New York, NY, 2013, Ch.~9, pp. 187--207.

\bibitem{b2003}
B.~Cockburn, Discontinous galerkin methods, Z. Angew. Math. Mech. 83 (2003)
  731--754.

\bibitem{bot08}
N.~G. Bean, M.~M. O'Reilly, P.~G. Taylor, Algorithms for the
  {L}aplace-{S}tieltjes transforms of first return times for stochastic fluid
  flows, Methodology and Computing in Applied Probability 10 (2009) 381--408.

\bibitem{riviere2008discontinuous}
B.~Riviere, Discontinuous {G}alerkin methods for solving elliptic and parabolic
  equations: {T}heory and implementation, SIAM, 2008.

\end{thebibliography}
\bibliographystyle{abbrv}

\end{document}